\documentclass[11pt]{article}
\usepackage{amssymb,amsmath,amsfonts,amsthm,mathtools,enumerate,color}
\usepackage[toc,page,title,titletoc,header]{appendix}
\usepackage{graphicx,subfig}

\graphicspath{ {./figures/} }

\usepackage{indentfirst}
\usepackage{multicol}
\usepackage{booktabs}
\usepackage{url}
\usepackage{epstopdf}  

\usepackage{hyperref}
\hypersetup{
    colorlinks=true, 
    linktoc=all,     
    linkcolor=blue,  
}
\numberwithin{equation}{section}

\newtheorem{Theorem}{Theorem}[section]
\newtheorem{Lemma}[Theorem]{Lemma}
\newtheorem{Proposition}[Theorem]{Proposition}
\newtheorem{Definition}[Theorem]{Definition}
\newtheorem{Corollary}[Theorem]{Corollary}

\newtheorem{Remark}{Remark}
\newtheorem{Assumption}{Assumption}
\newtheorem{Condition}{Condition}

  \def\nb{\nonumber}
\def \Vh0{\stackrel{\circ}{V}_h} \def\to{\rightarrow}
   \def\ul{\underline}
\def\Om{\Omega}  \def\om{\omega} 
 
\newcommand{\q}{\quad}

\def\l{\label}  \def\f{\frac}  \def\fa{\forall}
\def\b{\beta}  \def\a{\alpha} 
\def\del{\delta}
\def\eps{\varepsilon}
 \def\t{\times}  
\def\ms{\medskip}

  \def\x{{\bf x}} 
\def\y{{\bf y}}

\def\cA{\mathcal{A}}
\def\cB{\mathcal{B}}

\def\cE{\mathcal{E}}
\def\cF{\mathcal{F}}
\def\cH{\mathcal{H}}
\def\cI{\mathcal{I}}

\def\cM{\mathcal{M}}
\def\cP{\mathcal{P}}
\def\cQ{\mathcal{Q}}
\def\cS{\mathcal{S}}
\def\cT{\mathcal{T}}

\def\bA{{\textbf{A}}}

\def\N{{\mathbb{N}}}

\def\R{{\mathbb R}}
\def\Z{{\mathbb{Z}}}

\newcommand{\ex}{\mathbb{E}}

\newcommand{\tr}{\textnormal{tr}}

\newcommand{\lc}
{\mathrel{\raise2pt\hbox{${\mathop<\limits_{\raise1pt\hbox
{\mbox{$\sim$}}}}$}}}

\newcommand{\gc}
{\mathrel{\raise2pt\hbox{${\mathop>\limits_{\raise1pt\hbox{\mbox{$\sim$}}}}$}}}

\newcommand{\ec}
{\mathrel{\raise2pt\hbox{${\mathop=\limits_{\raise1pt\hbox{\mbox{$\sim$}}}}$}}}

\def\bb{\begin{equation}} \def\ee{\end{equation}}
\def\bbn{\begin{equation*}} \def\een{\end{equation*}}

\def\beqn{\begin{eqnarray}}  \def\eqn{\end{eqnarray}}

\def\beqnx{\begin{eqnarray*}} \def\eqnx{\end{eqnarray*}}

\def\bn{\begin{enumerate}} \def\en{\end{enumerate}}

\def\bd{\begin{description}} \def\ed{\end{description}}

\makeatletter
\newenvironment{tablehere}
  {\def\@captype{table}}
  {}
\newenvironment{figurehere}
  {\def\@captype{figure}}
  {}
\makeatother

\begin{document}

\title{Approximation schemes for mixed optimal stopping and control problems with nonlinear expectations and jumps}
\author{
Roxana Dumitrescu\thanks{Department of Mathematics, King's College London, United Kingdom
({\tt roxana.dumitrescu@kcl.ac.uk})}
\and
Christoph Reisinger\thanks{Mathematical Institute, University of Oxford, United Kingdom ({\tt christoph.reisinger@maths.ox.ac.uk, yufei.zhang@maths.ox.ac.uk})}
\and
Yufei Zhang\footnotemark[2]
}
\date{}

\maketitle

\noindent\textbf{Abstract.} 
We propose a class of numerical schemes for mixed optimal stopping and control 
of processes with infinite activity jumps and where the objective is evaluated by a
nonlinear expectation. 
Exploiting an approximation by switching systems, piecewise constant policy timestepping reduces the problem to nonlocal semi-linear equations with different control parameters, uncoupled over individual time steps, which we solve by fully implicit monotone approximations to the controlled diffusion and the nonlocal term, and specifically the Lax-Friedrichs scheme for the nonlinearity in the gradient. We establish a comparison principle for the switching system and demonstrate the convergence of the schemes, which subsequently gives a constructive proof for the existence of a solution to the switching system. Numerical experiments are presented for a recursive utility maximization problem to demonstrate the 
effectiveness of the new schemes. 


\ms

\medskip


\section{Introduction}\l{sec:introduction}
Classical Markovian mixed optimal stopping and control problems, where the target is to maximise the (linear) expectation of a payoff 
on a finite time horizon $T$, are  defined as\\
\bb\l{eq:linear_ctrl}
u(t,x)= \sup_{\tau}\sup_{\a} \ex^{t,x}\bigg[\int_t^\tau e^{-r(s-t)} f(\a_s,X^{\a,t,x}_s)\,ds+e^{-r(\tau-t)}\xi(\tau, X^{\a,t,x}_\tau)\bigg],
\ee
where $t\in [0,T]$ is the initial time of the control problem, 
 $\alpha$ is an admissible control process and $\tau$ is a stopping time,
and $X^{\a,t,x}$ is a controlled stochastic differential equation (SDE) of the form: 
$$
X^{\a,t,x}_s=b(\a_s,X^{\a,t,x}_s)\,ds+ \sigma (\a_s,X^{\a,t,x}_s)\,dW_s+\eta(\a_s,X^{\a,t,x}_s,e)\,\tilde{N}(ds,de),\; s\in [t,T]; \q X^{\a,t,x}_t=x.
$$
The positive constant $r$ denotes the discount rate, and the functions $\xi$ and $f$ represent the terminal payoff and the instantaneous reward function, respectively. 
Under certain regularity assumptions on the coefficients, one can demonstrate that the value function 
$u$ satisfies a nonlocal Hamilton-Jacobi-Bellman variational inequality (HJBVI) in the viscosity sense.

These results are extended in \cite{dumitrescu2016} to a setting where the linear expectation $\ex$ is replaced by a nonlinear expectation
$\cE^{\a, t,x}$ generated by a BSDE with jumps, 
\bb\l{eq:nonlinear_ctrl}
u(t,x)= \sup_{\tau}\sup_{\a} \cE^{\a, t,x}_{t,\tau}[\xi(\tau, X^{\a,t,x}_\tau)].
\ee


Such nonlinear expectations arise naturally in financial mathematics, for instance as models for American options in a  market with constrained portfolios \cite{karoui1997}, from recursive utility optimization problems \cite{chen2002}, and  for robust pricing and risk measures under  probability model uncertainty \cite{royer2006}. 
It has been demonstrated in \cite{dumitrescu2016} that under suitable assumptions the value function $u$ in \eqref{eq:nonlinear_ctrl} can be characterized  by the viscosity solution to a more complicated HJBVI \eqref{eq:hjbvi_1}, which involves an extra nonlinearity resulting from the nonlinear expectation:
\begin{align}\l{eq:hjbvi_1}
\min\big\{u(\x)-\zeta(\x),u_t+\inf_{\a\in \bA}\big(-L^\a u-f(\a,\x,u,(\sigma^\a)^T Du,B^\a u)\big)
\big\}=& \; 0,  \q \x\in [0,T]\t \R^d, \\
u(0,x)=& \; g(x), \q x\in \R^d, \nonumber
\end{align}
with $\x = (t,x)$, nonlocal operators $L^\a$ and $B^\a$, the driver $f$ of the BSDE, and given functions $\zeta$ and $g$, which we will specify in Section \ref{sec:formulation}.
Particularly, in the case where the driver is additive in $y$ and independent of $z$ and $k$, i.e., $f(\a,\x,y,z,k)\equiv f(\a,x)-ry$, the generalized control problem \eqref{eq:nonlinear_ctrl} reduces to the classical linear  expectation case \eqref{eq:linear_ctrl}, and \eqref{eq:hjbvi_1} reduces to an
HJB obstacle problem.\footnote{Note a slight abuse of notation where $f(\cdot,\cdot)$ is the same as in \eqref{eq:linear_ctrl}.}
As it is usually difficult to obtain analytic solutions of HJBVIs, it is necessary to design efficient and robust numerical methods for solving these fully nonlinear PIDEs.
\color{black}

We remark that, to the best of our knowledge, even for the case with linear expectations (i.e.\ $f$ in the special form from above), there is no published numerical scheme covering the generality of \eqref{eq:hjbvi_1}. 
However, 
there is a vast literature on monotone approximations for local HJB equations (see, e.g., \cite{crandall1984, barles2007, debrabant2012} and references therein)
and a number of works covering specific extensions. 
For instance, monotone finite-difference quadrature schemes are proposed in \cite{cont2005,biswas2010,biswas2010diff,biswas2017} for nonlocal HJB equations. We refer the reader also to \cite{dHalluin2004} for penalty approximations to nonlocal variational inequalities and to \cite{witte2011} for an application of policy iteration  together with penalization to solve HJB obstacle problems. 
Probabilistic methods for solving HJB equations (without jumps and optimal stopping) can be found, for example, in \cite{kharroubi2015}.

All the aforementioned 
PDE methods
solve \eqref{eq:hjbvi_1} (with $f(\a,\x,y,z,k)\equiv f(\a,x)-ry$) by the standard ``discretize, then optimize" approach, where one discretizes the operators in \eqref{eq:hjbvi_1}, and solves the resulting nonlinear discretized equations using policy iteration, or more generally semi-smooth Newton methods \cite{qi1993,ito2003}. 

However, this standard approach cannot be easily extended to nonlinear $f$ which is only assumed to be Lipschitz (and generally is not semi-smooth \cite{qi1993}), which
prevents a direct application of 
Newton-like solvers (see \cite{xu2017} for a special case of
the discrete optimization problem with $f(\a,\x,y,z,k)\equiv f(\a,x,y)$ differentiable and concave in $y$, and $\bA$ finite).

Moreover, 
at each step of policy iteration, one needs to identify the global optimal policy for each computational node. The nonlinear driver $f$ and other PDE coefficients may have sufficiently complicate nonlinearities in the control variable such that the only way to construct a  convergent algorithm is to discretize the admissible control set, and perform exhaustive search to determine the optimal policy at each node. 

Another approach to solve \eqref{eq:hjbvi_1} uses piecewise constant policy time stepping (PCPT) as in~\cite{reisinger2016}. 
It is based explicitly on a discrete approximation of the admissible set by a finite set, say with $J$ elements, and then defines a piecewise decoupled system of PDEs corresponding to these $J$ (constant) controls. The information from the different solutions is assembled at the end of each timestep by taking the pointwise maximum.

As a specific scheme for these semi-linear PDEs, we propose an implicit Euler time discretization, monotone (semi-Lagrangian) approximations for local diffusions,
 a monotone quadrature-based scheme for the nonlocal terms, and the Lax-Friedrichs scheme for the nonlinearity in the gradient.
The different solutions may be defined on different discretization grids by possibly high order monotonicity preserving interpolations. 
This approach not only avoids policy iteration, but also allows for an easier construction of convergent monotone schemes and an efficient parallel implementation
of the individual semi-linear PDEs.
Note that it is essential to obtain a monotone discretization, since it is well-known that non-monotone schemes may fail to converge or even converge to false ``solutions'' \cite{debrabant2012}.
By Godunov's Theorem \cite{godunov1959}, in general, one can expect a monotone scheme to be at most first-order accurate.

The main contributions of our paper are:
\begin{itemize}
\item We formulate our algorithm by approximating the solution of \eqref{eq:hjbvi_1} by the solution to a switching system with small switching cost. 
We shall establish a comparison principle for the switching system and demonstrate that as the switching cost tends to zero, the solution of the switching system converges to the viscosity solution of \eqref{eq:hjbvi_1}, which extends the results in \cite{biswas2010} to obstacle problems of switching systems and includes nonlinear drivers.
\item We discretize the switching system piecewise in time by fully implicit monotone approximations.
The convergence of the scheme is demonstrated, which subsequently gives a  constructive proof for the existence of a viscosity solution to the switching system. 
Our results extend the one obtained in \cite{reisinger2016} from the case of standard control problems. In contrast to there, PCPT leads to 
coupled semi-linear PDEs rather than linear PDEs due to the nonlinear expectations.
The optimal stopping right is treated as an additional control and included in the switching directly instead of the classical penalisation approach. 
\item
By truncation of the singular jump measure, we obtain a stochastic control and optimal stopping problem whose value function is shown to converge to the value function of the initial problem and which satisfies a HJBVI equation. Our result extends earlier ones obtained only in the case of a linear expectation without control and optimal stopping
(see e.g. \cite{cont2005}).
\item For practical implementations, we propose a Picard-type iteration for the efficient numerical solution without the need to invert the dense matrices resulting from the nonlocal terms.
\item
Numerical examples for a recursive utility maximization problem are included to investigate the convergence order of the scheme with respect to different discretization parameters. 
\end{itemize}

The remainder of this paper is organized as follows. In Section \ref{sec:formulation}, we introduce the Markovian mixed optimal stopping and control problem with nonlinear expectations, and characterize its value function as the viscosity solution of a nonlocal HJBVI. We then derive numerical schemes in Section \ref{sec:scheme} by approximating the HJBVI with a switching system, PCPT, and ultimately fully discrete monotone schemes.
Then we move on to the convergence analysis of our numerical schemes in Section \ref{sec:convergence}. 
Numerical examples for a recursive utility maximization problem are presented in Section \ref{sec:numerical} to illustrate the effectiveness of our algorithms. In the Appendix, we include a rigorous proof of the comparison principle for the switching system and some complementary results that are used in this article.  

\section{Problem formulation and preliminaries}\l{sec:formulation}
In this section, we formulate the mixed optimal stopping and control problem with nonlinear expectation and introduce the connection between such problems and HJBVIs, which is crucial for the subsequent developments. 
We start with some useful notation that is needed frequently in the rest of this work. 

We write by $T > 0$ the terminal time, and by $(\Om, \cF, P )$ a complete probability space, in which two mutually independent processes, a $d$-dimensional Brownian motion $W$ and a Poisson random measure $N(dt,de)$ with compensator $\nu(de)dt$, are defined. We assume $\nu$ is a $\sigma$-finite measure on $E\coloneqq\R^n\setminus\{0\}$ equipped with its Borel field $\cB(E)$ and satisfies
\bb\l{eq:nu_int}
\int_E (1 \wedge |e|^2)\,\nu(de) <\infty.
\ee
 We denote by $\ex$ the usual expectation operator with respect to the measure $P$.
 
For any given $t \in [0,T]$, we define the $t$-translated Brownian motion $W^t:=(W_s-W_t)_{s \geq t}$ and the $t$-translated Poisson random measure $N^t:=N(]t,s],\cdot)_{s \geq t}$. We denote by $\tilde{N}^t(dt,de)=N^t(dt,de)-\nu(de)dt$ the compensated process of $N^t$,
  and by $\mathbb{F}^t=\{\cF^t_s\}_{s\in[t,T]}$ be the  filtration generated by $W^t$ and $N^t$ augmented by the $P$-null sets.

  Furthermore, we introduce several spaces: $L^2_\nu$ is the  space of Borel functions $l:E\to \R$ with $\|l\|_\nu^2\coloneqq \int_E|l(e)|^2\,\nu(de)<\infty$; 
$\mathbb{H}^2_t$ (resp. $\mathbb{H}^2_{t,\nu}$) is the  space of $\R^{d}$-valued (resp.\ real-valued) $\mathbb{F}^t$-predictable processes $(\pi_s)$ (resp.\ $(l_s(\cdot)$) with $\ex\int_t^T |\pi_s|^2\,ds<\infty$ (resp.\ $\ex\int_t^T \|l_s\|_{L^2_\nu}^2\,ds<\infty$); $\cS^2_t$ is the space of real-valued $\mathbb{F}^t$-adapted c\`{a}dl\`{a}g processes $(\psi_s)$ with $\ex[\sup_{t \leq s \leq T} \psi_s^2]<\infty$.

We now proceed to introduce the control problem of interest. For each $t\in [0,T]$,
let $\cA_t^t$ be a set of admissible controls, which are $\mathbb{F}^t$-predictable processes $(\a_s)_{s\in[t,T]}$ valued in a compact set $\mathbf{A}$, and $\mathcal{T}_t^t$ be the set of $\mathbb{F}^t$-stopping times which take values in $[t,T]$.
For any given  initial state $x\in \R^d$, and control $\a\in \cA_t^t$, we consider the controlled jump-diffusion process 
 $(X^{\a,t, x}_s)_{ t \le s\le T}$ satisfying the following SDE: for each $s\in [t,T]$,
\bb\l{eq:sde}
X^{\a,t,x}_s=x+\int_t^s b(\a_v,X^{\a,t,x}_v)\,dv+\int_t^s \sigma (\a_v,X^{\a,t,x}_v)\,dW^t_v+\int_t^s\int_E \eta(\a_v,X^{\a,t,x}_v,e)\,\tilde{N}^t(dv,de),
\ee
where $b$, $\eta\in \R^d$ and $\sigma\in \R^{d\t d}$ are given {measurable} functions. We remark that although our analyses are performed only for   jump-diffusion processes with time-homogenous coefficients,  similar results are valid for controlled dynamics with time-dependent coefficients.

The performance of the control problem, depending on $\a$, is  evaluated  by a nonlinear expectation induced by a BSDE with a  controlled driver $f(\a_s,s,X_s^{\a,t,x},y,z,k)$. 
That is, for any given stopping time $\tau\in \mathcal{T}_t^t$ and  any bounded Borel function $\xi$, we define the nonlinear expectation 
$$
\cE^{\a,t, x}_{t,\tau}[\xi(\tau,X^{\a,t,x}_\tau)]\coloneqq Y^{\a, \tau,t, x}_{t},
$$
where the process $(Y^{\a, \tau, t,x}_{s})_{s\le \tau}$ is a solution in $\cS_t^2$ of the following BSDE:
for each $s\in [t,\tau]$,
\bb\l{eq:bsde}
\begin{cases}
-Y^{\a,t,x}_{s,\tau}=f(\a_s,s,X_s^{\a,t,x},Y^{\a,t,x}_{s,\tau},Z^{\a,t,x}_{s,\tau},K^{\a,t,x}_{s,\tau})ds-Z^{\a,t,x}_{s,\tau}dW^t_s-\int_E K^{\a,t,x}_{s,\tau}\,\tilde{N}^t(ds,de),\\
Y^{\a,t,x}_{\tau,\tau}=\xi(\tau,X^{\a,t,x}_\tau),
\end{cases}
\ee
and $(Z^{\a,t,x}_{s,\tau})$, $(K^{\a,t,x}_{s,\tau})$ are two associated processes, if they exist, lying in $\mathbb{H}_t$ and $\mathbb{H}^\nu_t$, respectively. 

Now we are ready to state the generalized mixed optimal stopping and control problem.  
For each initial 
time $t\in[0,T]$ and initial state $x\in \R^d$, we consider the following value function:
\bb\l{eq:nonlinear_control}
u(t,x)= \sup_{\tau\in \cT_t^t}\sup_{\a\in \cA_t^t} \cE^{\a, t,x}_{t,\tau}[\xi(\tau, X^{\a,t,x}_\tau)],
\ee
subject to the controlled SDE \eqref{eq:sde}, where $\xi$ is the terminal position given by
$$
\xi(\tau,X^{\a,t,x}_\tau)=\zeta(\tau, X^{\a,t,x}_\tau)1_{t\le \tau<T}+g(X^{\a,t,x}_T)1_{\tau=T},
$$
for some reward functions $\zeta$ and $g$. Note that the value function of our control problem is constant up to a $P$-null set. Throughout this work, we shall perform the analysis under the following standard assumptions on the coefficients:
\begin{Assumption}\l{assum:coeff}
{The set of control values $\bA$ is compact} and the driver $f$ {is a measurable function} of the form 
$f(\a,s,x,y,z,k)\coloneqq \hat{f}(\a,s,x,y,z,\int_E k(e)\gamma(x,e)\,\nu(de))1_{s\ge t}$ for some functions 
$\hat{f}$ and $\gamma$. 
Moreover, there exists a constant $C>0$ such that for any $\a,\a'\in \bA$,  $t\in[0,T]$, $e\in E$, $x,x'\in \R^d$, $u,v\in \R$, $p,q\in \R^d$, $k,k'\in \R$, we have 
\begin{enumerate}[(1)]
\item $|b(\a,x)-b(\a',x')|+|\sigma(\a,x)-\sigma(\a',x')|\le C(|x-x'|+|\alpha-\alpha'|)$;

\item $|\eta(\a,x,e)-\eta(\a',x',e)|\le C(|x-x'|+|\a-\a'|)(1 \wedge |e|)$ and $|\eta(\a,x,e)|\le C(1 \wedge |e|)$;
\item  $|\gamma(x,e)-\gamma(x',e)|\le C|x-x'|(1 \wedge |e|^2)$; $|\gamma(x,e)|\le C(1 \wedge |e|)$ and $\gamma(x,e)\ge 0$;
 \item   $\hat{f}:\bA\t [0,T]\t \R^d\t\R\t \R^{d}\t\R\to \R$ is continuous in $t$ and admits the  properties:
\begin{enumerate}
\item (Boundedness.) $|\hat{f}(\a,t,x,0,0,0)|\le C$;
\item (Monotonicity.) $\hat{f}(\a,t,x,v,p,k)-\hat{f}(\a,t,x,u,p,k)\ge C(u-v)$ when $u\ge v$, and $k\to \hat{f}(\a,t,x,u,p,k)$ is non-decreasing in $k$;
\item (Lipschitz continuity.) $|\hat{f}(\a,t,x,u,p,k)-\hat{f}(\a',t,x,v,q,k')|\le C(|\a-\a'|+|u-v|+|p-q|+|k-k'|)$;
\item for each $R>0$, there exists a continuous function $m_R:\R_+\to \R_+$ with $m_R(0)=0$ and 
$$
|\hat{f}(\a,t,x,u,p,k)-\hat{f}(\a,t,x',u,p,k)|\le m_R(|x-x'|(1+|p|)), \q |x|,|x'|,|u|\le R;
$$
\end{enumerate}
\item the function $\zeta:[0,T]\t\R^d\to \R$ is continuous in $t$, and we have $|g(x)-g(x')|+|\zeta(t,x)-\zeta(t,x)|\le C|x-x'|$, 
$|g(x)|+|\zeta(t,x)|\le C$,  and $g(x)\ge \zeta(0,x)$.
\end{enumerate}
\end{Assumption}
Assumption \ref{assum:coeff} is the same as that made in \cite{dumitrescu2016a}.

{Note that under   assumptions $(1), (2)$, equation \eqref{eq:sde} admits a unique solution. Assumptions  $(3), (4.a), (4.c), (5)$ quarantee the existence and the uniqueness of the solution of equation \eqref{eq:bsde}}. Consequently, the generalized mixed control problem is well-defined.
For notational convenience, in the sequel, we will write $\hat{f}$ as $f$, and denote by $\psi^\a$ a generic function $\psi$ with control-dependence.

The rest of this section is devoted to the equivalence between the mixed control problem and a generalized nonlocal HJBVI.
Specifically, we now consider a Hamilton-Jacobi-Bellman variational inequality of the following form: 
\begin{align}\l{eq:hjbvi}
0&=F(\x,u,Du,D^2u,\{K^\a u\}_{\a\in\bA},\{B^\a u\}_{\a\in\bA})\\
&=\begin{cases}
\min\big\{
u(\x)-\zeta(\x),u_t+\inf_{\a\in \bA}\big(-L^\a u-f(\a,\x,u,(\sigma^\a)^T Du,B^\a u)\big)
\big\},  & \x\in \cQ_T ,\nb\\
u(\x)-g(x), & \x\in \{0\}\t\R^d,\nb
\end{cases}
\end{align}
where $\cQ_T= (0,T]\t \R^d$, $\x=(t,x)$ contains both the time $t$ and the spatial coordinate $x\in \R^d$, and the nonlocal operators $L^\a\coloneqq A^\a+K^\a$ and $B^\a$ satisfy, for $\phi\in C^{1,2}(\bar{\cQ}_T)$:
\begin{align}
A^\a\phi(\x)&=\f{1}{2} \tr(\sigma^\a(x)(\sigma^\a(x))^TD^2\phi(\x))+b^\a(x) \cdot D\phi(\x),\l{eq:a}\\
K^\a\phi(\x)&=\int_{{E}}\big(\phi(t,x+\eta^\a(x,e))-\phi(\x)-\eta^\a(x,e)\cdot D\phi(\x)\big)\,\nu(de),\l{eq:k}\\
B^\a\phi(\x)&=\int_{{E}}\big(\phi(t,x+\eta^\a(x,e))-\phi(\x)\big)\gamma(x,e)\,\nu(de),\l{eq:b}
\end{align}
where $E=\R^n\setminus\{0\}$ is defined at the beginning of Section \ref{sec:formulation}
and
the nonlocal operators $K^\a u$ and $B^\a u$ are well-defined under Assumption \ref{assum:coeff}. 
%
%
%
%
%
%

We emphasize that since the  matrix $\sigma^\a(\sigma^\a)^T$ is only assumed to be nonnegative definite, both the diffusion   coefficient  $\sigma^\a(\sigma^\a)^T$ and the jump intensity $\eta$ of \eqref{eq:hjbvi} are allowed to vanish at some points. Consequently, there is no Laplacian smoothing from the second-order differential operator nor fractional Laplacian smoothing from the nonlocal operator to this degenerate equation \eqref{eq:hjbvi}. Therefore, in general, this HJBVI will not admit classical solutions, and we shall interpret the equation in the following viscosity sense  based on semi-continuous envelopes of the equation \cite{barles1997,reisinger2016}. 

\begin{Definition}[Viscosity solution of  HJBVI]
An upper (resp.\ lower) semicontinuous function $u$ is said to be a viscosity subsolution (resp.\ supersolution) of \eqref{eq:hjbvi} if and only if 
for any point $\x_0$ and for any $\phi\in C^{1,2}(\bar{\cQ}_T)$ such that 
$\phi(\x_0)=u(\x_0)$ and $u-\phi$ attains its global  maximum (resp.\ minimum) at $\x_0$, one has
\begin{align*}
&F_*(\x_0,u(\x_0),D\phi(\x_0),D^2\phi(\x_0),\{K^\a \phi(\x_0)\}_{\a\in\bA},\{B^\a \phi(\x_0)\}_{\a\in\bA})\le 0,
\\
\big(resp. \q &F^*(\x_0,u(\x_0),D\phi(\x_0),D^2\phi(\x_0),\{K^\a \phi(\x_0)\}_{\a\in\bA},\{B^\a \phi(\x_0)\}_{\a\in\bA})\ge 0
\big).
\end{align*}
%
%
A continuous function is a viscosity solution of the  HJBVI \eqref{eq:hjbvi} if it is both a a viscosity sub- and supersolution.
\end{Definition}

Under Assumption \ref{assum:coeff}, the HJBVI \eqref{eq:hjbvi} is well-posed in the class of 
bounded continuous functions (see \cite{dumitrescu2015,dumitrescu2016}). The unique  viscosity solution of  \eqref{eq:hjbvi} (after a change of time variable) can be further characterized as the optimal value function \eqref{eq:nonlinear_control} of the mixed control problem. In other words, to obtain the optimal value function for all initial times $t$ and initial states $x$, it is equivalent to design effective numerical schemes to solve \eqref{eq:hjbvi}.

Moreover, a strong comparison principle holds for the  HJBVI \eqref{eq:hjbvi}, the proof of which is similar to that in \cite{dumitrescu2015} (without controls) and hence omitted.
In particular, if $U$ is a bounded viscosity subsolution and $V$ is a bounded viscosity supersolution to \eqref{eq:hjbvi} with $U(0,\cdot)\le V(0,\cdot)$, we have $U(\x)\le V(\x)$ for all $\x\in \bar{\cQ}_T$. 

\section{Construction of  numerical schemes}\l{sec:scheme}
In this section, we will design numerical schemes for solving HJBVI \eqref{eq:hjbvi}. 
We carry out the following string of  approximations to construct our numerical algorithm:
\begin{itemize}
\item truncation of the singular jump measure (equation (\ref{eq:hjbvi_trun}) and Appendix \ref{sec:jumps});
\item approximation of the control set with a finite set (equation (\ref{eq:hjbvi_dc}) and Theorem \ref{theo:disc_cont});
\item approximation of the discretized control problem with a switching system (equation (\ref{eq:hjbvi_s}) and Theorem \ref{thm:conv_s});
\item discretization in time and space (equations (\ref{eq:Un+1/2}, \ref{eq:PCCT}) and Theorem \ref{thm:conv_numerical}).
\end{itemize}

We  start the derivations of our schemes by approximating the singular measure $\nu$ with a  truncated non-singular measure and a modified diffusion coefficient as suggested in \cite{cont2005}. This can be done by introducing an approximative jump-diffusion dynamics and an approximative backward SDE (see Appendix \ref{sec:jumps}).
More precisely, for any given $\eps>0$, let us define the truncated  measure $\nu_{\eps}(de)=1_{|e|>\eps}\nu(de)$ and the modified diffusion coefficient $\tilde{\sigma}^{\a}(x)$ such that $\tilde{\sigma}^{\a}_{ij}(x)=\sigma^{\a}_{ij}(x)$ for $ i\not=j$ and 
\bb\l{eq:small_diff}
\tilde{\sigma}^{\a}_{ii}(x)=\bigg((\sigma^{\a}_{ii}(x))^2+\int_{|e|<\eps}|\eta^{\a}_i(x,e)|^2\, \nu(de)\bigg)^{1/2}, \q i=1,\ldots, d, \; x\in \R^d.
\ee
We further introduce the modified local operator $A^{\a}_{\eps}$ as: 
\bb\l{eq:diff_modi}
A^{\a}_{\eps}\phi(\x)\coloneqq \f{1}{2} \tr(\tilde{\sigma}^\a(x)(\tilde{\sigma}^\a(x))^TD^2\phi(\x)),\q \phi\in C^{1,2}([0,T]\t\R^d),
\ee
and   truncated nonlocal operators $K^{\a}_{\eps}$ and $B^{\a}_{\eps}$ by replacing $\nu$ with $\nu_{\eps}$ in \eqref{eq:k} and \eqref{eq:b}, respectively.

With these operators in hand, we consider the following modified HJBVI:
\begin{align}\l{eq:hjbvi_trun}
0&=F^\eps(\x,u,Du,D^2u,\{K_\eps^\a u\}_{\a\in\bA},\{B_\eps^\a u\}_{\a\in\bA})\\
&=\begin{cases}
\min\big\{
u-\zeta,u_t+\min_{\a\in \bA}\big(-L_\eps^\a u-f^{\a}(\x,u,(\tilde{\sigma}^\a)^T Du,B_\eps^\a u)\big)
\big\},  & \x\in \cQ_T ,\nb\\
u(\x)-g(x), & \x\in \{0\}\t\R^d,\nb
\end{cases}
\end{align}
where we have $L^{\a}_{\eps}\phi\coloneqq A^{\a}_{\eps}\phi+K^{\a}_{\eps}\phi$ for any $\phi\in C^{1,2}([0,T]\t\R^d)$. These modified coefficients clearly satisfy Assumption \ref{assum:coeff}, and hence   \eqref{eq:hjbvi_trun} is well-posed in the viscosity sense. 

%

\begin{Remark}
In  Appendix \ref{sec:jumps}, we provide an alternative interpretation of the above approximation by identifying the viscosity solution of  \eqref{eq:hjbvi_trun} as the  value function of a mixed  control problem in terms of  modified SDE and BSDE. This characterization further enables us to establish the convergence of this approximation through a probabilistic argument. 

\end{Remark}

We then approximate the admissible control set in \eqref{eq:hjbvi_trun} by a finite set. More precisely, for a finite subset $\bA_\delta$ of the compact set $\bA$
  such that
$$
\max_{\a\in \bA}\min_{\tilde{\a}\in \bA_\delta}|\a-\tilde{\a}|<\delta,
$$
we introduce the finite control  HJBVI by
\begin{align}\l{eq:hjbvi_dc}
0&=F^{\eps,\delta}(\x,u,Du,D^2u,\{K_\eps^\a u\}_{\a\in\bA_\delta},\{B_\eps^\a u\}_{\a\in\bA_\delta})\\
&=\begin{cases}
\min\big\{
u-\zeta,u_t+\min_{\a\in \bA_\delta}\big(-L_\eps^\a u-f^{\a}_{\eps}[u]\big)
\big\},  & \x\in \cQ_T ,\nb\\
u(\x)-g(x), & \x\in \{0\}\t\R^d,\nb
\end{cases}
\end{align}
where we denote for simplicity the nonlinear function $f^{\a}_{\eps}[u]\coloneqq f^{\a}(\x,u^\delta,(\tilde{\sigma}^\a)^T Du^\delta,B_\eps^\a u^\delta)$. Since \eqref{eq:hjbvi_dc} is a special case of \eqref{eq:hjbvi} with a finite admissible set, it is clear that \eqref{eq:hjbvi_dc}  admits a unique bounded viscosity solution.

Next, we approximate the finite control equation \eqref{eq:hjbvi_dc} by a switching system (\cite{barles2007,biswas2010}).  
Suppose the finite control set is given by  $\bA_\delta=\{\a_1,\a_2,\ldots, \a_J\}$. We denote by $U_j^{\eps,\delta,c}$, $j=1,\ldots, J$  the solution of the following system of HJB equations:
\begin{align}\l{eq:hjbvi_s}
0&=F_j^{\eps,\delta,c}(\x,U_j, DU_j, D^2U_j,\{K_\eps^\a u\}_{\a\in\bA_\delta},\{B_\eps^\a u\}_{\a\in\bA_\delta},\{U_k\}_{k\not =j})\nb\\
&=\begin{cases}
\min\bigg[U_j-\zeta,\; \min\big(U_{j,t}-L_\eps^{\a_j} U_j-f_\eps^{\a_j}[U_j]; \; U_{j}-\cM_jU\big)\bigg],  &\x\in \cQ_T,\\
U_j(\x)-g(x),  & \x\in \{0\}\t\R^d,
\end{cases}
\end{align}
where we define $\cM_jU\coloneqq \max_{k\not =j}\big(U_k-c\big)$ for any $c>0$.
The positive switching cost  is needed for the well-posedness of the switching system \eqref{eq:hjbvi_s}.


We now proceed to introduce a discrete approximation to the switching system based on the idea of piecewise constant policy timestepping. Define a set of nodes $\{x_{j,i}\}$ and timesteps $t^n$, with discretization parameters $h$ and $\Delta t$, i.e., 
\bb
\max_{1\le j\le J,x\in \R^d}\min_{i}|x-x_{j,i}|=h,\q \max_n(t^{n+1}-t^n)=\Delta t.
\ee
By parameterizing the grid $\Omega_{j,h}=\{\x_{j,i}\}_i$ with the control index $j$, we are allowing the usage of different discretization grids for different controls. 
We denote by $U^{n}_{j,i}$ the discrete approximation to $U_{j}$ at the  point $\x^n_{j,i}=(t^n,x_{j,i})$, and extend it to 
the computational domain by interpolation.

Let $L^{\a_j}_{\eps,h}$ and $f^{\a_j}_{\eps,h}$ be the discrete form of the operators $L_\eps^{\a_j}$ and $f_\eps^{\a_j}$, respectively. We discretize \eqref{eq:hjbvi_s} on the grid $\Omega_{j,h}$ with the uniform time partition $t^{n+1}-t^n=\Delta t$ by performing piecewise constant policy timestepping and applying the constraints at the beginning of a new timestep,
\begin{align}
U^{n+\f{1}{2}}_{j,i}&=\max\big[\zeta_i^{n+1}, \, U^{n}_{j,i},\,\max_{k\not= j}(\tilde{U}^{n}_{k,i(j)}-c)\big)\big],\l{eq:Un+1/2} \\
U^{n+1}_{j,i}-\Delta t\big(L_{\eps,h}^{\a_j} U_{j,i}^{n+1}+f_{\eps,h}^{\a_j}[U_{j,i}^{n+1}]\big)&=U^{n+\f{1}{2}}_{j,i},\q\q j=1,\ldots, J, \l{eq:PCCT}
\end{align}
where $\tilde{U}^{n}_{k,i(j)}$ is the value of the interpolant of $\{U^{n}_{k,l}\}_{l\in \Omega_{k,l}}$ at the $i$-th point of the grid $ \Omega_{j,h}$.

Now  by rearranging the terms of \eqref{eq:PCCT}, we obtain the following numerical scheme: for $\x_{j,i}^n\in \cQ_T$ and $j=1,\ldots, J$,
\begin{align}
0=&G_j(\x^{n+1}_{j,i},h,U^{n+1}_{j,i}, \{U^{b+1}_{j,a}\}_{\textnormal{$(a,b)\!\not =\!(i,n)$}}, \{
\tilde{U}^n_k\}_{k\not =j})\nb\\
=&\min\bigg[U^{n+1}_{j,i}-\zeta_i^{n+1}-\Delta t\big(L_{\eps,h}^{\a_j} U_{j,i}^{n+1}+f_{\eps,h}^{\a_j}[U_{j,i}^{n+1}]\big),\; 
\f{U^{n+1}_{j,i}-U^{n}_{j,i}}{\Delta t}-L_{\eps,h}^{\a_j} U_{j,i}^{n+1}-f_{\eps,h}^{\a_j}[U_{j,i}^{n+1}], \nb\\
&\hspace{1cm} U^{n+1}_{j,i}-\max_{k\not= j}(\tilde{U}^{n}_{k,i(j)}-c)
-\Delta t\big(L_{\eps,h}^{\a_j} U_{j,i}^{n+1}+f_{\eps,h}^{\a_j}[U_{j,i}^{n+1}]\big)
\bigg]. \l{eq:scheme}
\end{align}

As seen from \eqref{eq:scheme}, performing switching at the beginning
of a new timestep introduces two additional terms to both the switching part and the obstacle part of the equation, which will not appear in a straightforward discretization of the switching system \eqref{eq:hjbvi_s}. However, we will demonstrate in Section \ref{sec:implicit} that these  terms vanish as $\Delta t, h\to 0$, and consequently our scheme \eqref{eq:scheme} forms a consistent approximation to \eqref{eq:hjbvi_s}. 

For notational simplicity, we label our approximations only by $h$ and assume in the sequel that $\Delta t$ is a given function of $h$ with $\Delta t\to 0$ as $h\to 0$.  

\begin{Remark}\label{rem:controls}
To determine the optimal stopping strategy and the optimal controls,
one can simply compare values of  the obstacle and all components of the switching system at each grid point.
As we will see in Section \ref{sec:conv_switching}, each component of the switching system converges to the solution of the HJBVI \eqref{eq:hjbvi} as the discretization parameters tend to 0; therefore, although we have no guarantee for the convergence of the control approximation, this 
numerically found control is close to optimal 
for the mixed control problem \eqref{eq:nonlinear_control}.
\end{Remark}


We now describe in detail  how we perform spatial discretizations for $L^{\a}_{\eps}$ and $B^{\a}_{\eps}$ to construct a monotone discrete  operator $L^{\a}_{\eps,h}$ and $f^{\a}_{\eps,h}$, for a fixed control parameter $\a\in\{\a_1,\ldots,\a_J\}$.
To simplify the presentation, we  
consider a piecewise linear or multilinear interpolation $\cI_h$ on a uniform spatial grid $h\Z^d$. That is,
\bb\l{eq:lininterp}
\cI_h[\phi](x)=\sum_{m\in \Z^d}\phi(x_m)\omega_m(x;h),\q x\in\R^d,
 \ee
for  the standard ``tent functions''  $\om_m$ satisfying $0\le  \om_m(x;h)\le 1$, $\om_m(x_i;h)=\delta_{mi}$, $\sum_m\om_m=1$, $\textrm{supp}\,\om_m\subset B(x_m,2h)$, and $|D\om_m|\le C/h$.

We start with the nonlocal terms. 
The definition of $K^{\a}_{\eps}$ gives 
\begin{align*}
K^\a_{\eps}\phi(\x)&=\int_{|e|\ge \eps}\big(\phi(t, \x+\eta^\a(x,e))-\phi(x)-\eta^\a(x,e)\cdot D\phi(\x)\big)\,\nu(de)\\
&=\int_{|e|\ge \eps}\big(\phi(t,x+\eta^\a(x,e))-\phi(x)\big)\,\nu(de)+\int_{|e|\ge \eps}-\eta^\a(x,e)\,\nu(de)\cdot D\phi(\x)\\
&\coloneqq K^{\a,1}_{\eps}\phi(\x)+b^{\a}_{\eps}(x)\cdot D\phi(\x).
\end{align*}
Then, by replacing the integrands by their monotone interpolants (c.f. \cite{biswas2010diff}), we derive the following approximations for $K^{\a,1}_{\eps}$ and $B^{\a}_{\eps}$
(where we have dropped the mesh index $j$ in $x$ for simplicity): 
\begin{align}
K^{\a,1}_{\eps,h}\phi(t^n,x_i)&\coloneqq\int_{|e|\ge \eps}\cI_h[\phi(t^n,x_i+\cdot)-\phi(t^n,x_i)](\eta^\a(x_i,e))\,\nu(de)\nb\\
&=\sum_{m\in \Z^d}\kappa^{\a,n}_{h,m,i}[\phi(t^n,x_i+x_m)-\phi(t^n, x_i)] \l{eq:K_mono}\\
B^{\a}_{\eps,h}\phi(t^n,x_i)&\coloneqq\int_{|e|\ge \eps}\cI_h[\phi(t^n,x_i+\cdot)-\phi(t^n,x_i)](\eta^\a(x_i,e))\gamma(x_i,e)\,\nu(de)\nb\\
&=\sum_{m\in \Z^d}\b^{\a,n}_{h,m,i}[\phi(t^n,x_i+x_m)-\phi(t^n, x_i)],\l{eq:B_mono}
\end{align}
with the coefficients 
\bb\l{eq:k&b_coeff}
\kappa^{\a,n}_{h,m,i}\coloneqq \int_{|e|\ge \eps}\om_m(\eta^\a(x_i,e);h)\,\nu(de),\q 
\b^{\a,n}_{h,m,i}\coloneqq \int_{|e|\ge \eps}\om_m(\eta^\a(x_i,e);h)\gamma(x_i,e)\,\nu(de),
\ee
which are well-defined and nonnegative, and consequently result in monotone approximations.
{These coefficients can be efficiently evaluated by using quadrature rules with positive weights, such as Gauss methods of appropriate order \cite{biswas2010diff}.}

We then turn to discretizing the local terms. 
We introduce the modified drift  by
$\tilde{b}^\a(x)=b^\a(x)+b^\a_{\eps}(x)$,
and write the modified diffusion $\tilde{\sigma}^\a=(\tilde{\sigma}^\a_1,\ldots, \tilde{\sigma}^\a_d)$,
where $\tilde{\sigma}_l^\a$, $l=1,\ldots, d$ is the $l$-th column of $\tilde{\sigma}^\a$
 defined  in \eqref{eq:small_diff}.

With these modified coefficients, we are ready to construct the following approximations of the local operators: for any $k>0$,
\begin{align*}
\f{1}{2}\tr(\tilde{\sigma}^\a(x)(\tilde{\sigma}^\a(x))^TD^2\phi(x))&
\approx\f{1}{2}\sum_{l=1}^{d}\f{\cI_h[\phi](x+k\tilde{\sigma}^\a_l)-2\cI_h[\phi](x)+\cI_h[\phi](x-k\tilde{\sigma}^\a_l)}{k^2}\\
\tilde{b}^\a(x) \cdot D\phi&\approx 
\f{1}{2}\f{\cI_h[\phi](x+k^2\tilde{b}^\a)-2\cI_h[\phi](x)+\cI_h[\phi](x+k^2\tilde{b}^\a)}{k^2},
\end{align*}
which, by using \eqref{eq:lininterp} and the fact that $\sum_m\om_m=1$, can be further written in the discrete monotone form
\begin{align}\l{eq:A_mono}
A^{\a}_{\eps,h,k}\phi(t^n,x_i)=\sum_{m\in \Z^d}d^{\a,n}_{h,k,m,i}[\phi(t^n,x_m)-\phi(t^n,x_i)],
\end{align}
with non-negative coefficients
\bb\l{eq:d_coeff}
d^{\a,n}_{h,k,m,i}=\f{1}{2}\sum_{l=1}^{d}\f{\om_m(x_i+k\tilde{\sigma}^\a_l(x_i);h)+\om_m(x_i-k\tilde{\sigma}^\a_l(x_i);h)}{k^2}+\f{\om_m(x_i+k^2\tilde{b}^\a(x_i);h)}{k^2}\ge 0.
\ee
The approximation to the local operator $A^{\a}_{\eps}$ falls into the class of semi-Lagrangian schemes (see e.g.\ \cite{debrabant2012}), and provides a consistent monotone approximation for possibly degenerate, non-diagonally dominant diffusion coefficients.


Before presenting our fully discrete scheme, we shall point out  that by considering a truncated problem,
one can without loss of generality assume that $\sigma^\a$ is   bounded, which consequently implies the Hamiltonian
$$
\bar{f}^{\a}(\x,u, p,k)\coloneqq f^{\a}(\x,u,(\tilde{\sigma}^\a(x))^T p,k)
$$
 is Lipschitz continuous with respect to $p$. Indeed, suppose $\sigma^\a$ is unbounded, then for any given $\mu>0$, we define the cut-off function
$$
\xi_\mu:\R^d\to [0,1],\q \textnormal{$\xi_\mu\equiv 1$ for $|x|<\f{1}{\mu}$, and $\xi_\mu\in C_0^\infty(\R^d)$},
$$
and consider a truncated HJBVI \eqref{eq:hjbvi_dc} by replacing $\sigma^\a$ with the bounded diffusion coefficient $\sigma_\mu^\a(x)\coloneqq \xi_\mu(x)\sigma^\a(x)$.
Using the fact that $1-\xi_\mu\to 0$ uniformly on compact sets as $\mu\to 0$, one can easily prove  this additional approximation is consistent with \eqref{eq:hjbvi_dc}, and hence its viscosity solution converges to the solution to \eqref{eq:hjbvi_dc} uniformly on bounded sets.

The Lipschitz continuous Hamiltonian enables an approximation 
by the implicit Lax-Friedrichs numerical flux \cite{crandall1984}. For each $l=1,\ldots, d$, we denote by $\Delta ^{(l)}_+ U^n_{j,i}$ (resp.\ $\Delta ^{(l)}_- U^n_{j,i}$) the one-step forward (resp.\ backward) difference operator  along the $l$-th coordinate, and by $\Delta U^{n}_{j,i}=(\Delta ^{(1)}_+ U^n_{j,i}+\Delta ^{(1)}_- U^n_{j,i},\ldots, \Delta ^{(d)}_+ U^n_{j,i}+\Delta ^{(d)}_- U^n_{j,i})^T$ the central difference operator at the grid point $\x_{j,i}^n$. 
Then for any $\theta>0$, the Lax-Friedrichs numerical flux 
is given
for any $(\x^{n}_{j,i}, u, k)\in \Omega_{j,h}\t \R\t \R$ by
\bb\l{eq:numerical_flux}
\tilde{f}^{\a}(\x^{n}_{j,i},u,  \Delta U^{n}_{j,i},k)\coloneqq \bar{f}^{\a}(\x^{n}_{j,i},u,  \f{\Delta U^{n}_{j,i}}{2h},k)
 +\sum_{l=1}^d \f{\theta}{\lambda}\bigg(\f{\Delta ^{(l)}_+ U^n_{j,i}-\Delta ^{(l)}_- U^n_{j,i}}{h}\bigg),
\ee
where we define $\lambda=\Delta t/h$.

A fully implicit time discretisation is finally given by
\begin{align}\l{eq:implicit}
\begin{split}
U^{n}_{j,i}-\Delta t\big(A^{\a}_{\eps,h,k}U^{n}_{j,i}
+K^{\a,1}_{\eps,h}U^{n}_{j,i}+
 \tilde{f}^{\a}(\x^{n}_{j,i},U^{n}_{j,i},  \Delta U^{n}_{j,i},B^{\a}_{\eps,h}U^{n}_{j,i})
 \big)-U^{n-\f{1}{2}}_{j,i}=0.
\end{split}
\end{align}
Substituting \eqref{eq:Un+1/2} into \eqref{eq:implicit}, one can reformulate this implicit scheme in its equivalent form \eqref{eq:scheme}.

We end this section with a remark about the implementation of the implicit scheme \eqref{eq:implicit}. To avoid solving linear systems with the dense matrices resulting from the discretization of the nonlocal operators, we write the solution to \eqref{eq:implicit} as the fixed point of a sparse contraction mapping $T$,
such that sufficient accuracy is achieved in practice by a few fixed point iterations.

Given bounded functions $U^{n-\f{1}{2}}_{j}$ and $U^{n,(k)}_{j}$, we define the following mapping $T$ on  $\ell^\infty (\Z^d)$, i.e., the Banach space of bounded functions on  $h\Z^d$ employed with the sup-norm $|\cdot|_0$:
\bb\l{eq:contract_implement}
(1-\Delta t A^{\a}_{\eps,h,k})(TU^{n,(k)}_{j,i})
=\Delta t \big(K^{\a,1}_{\eps,h}U^{n,(k)}_{j,i}+
 \tilde{f}^{\a}(\x^{n}_{j,i},U^{n,(k)}_{j,i},  \Delta U^{n,(k)}_{j,i},B^{\a}_{\eps,h}U^{n,(k)}_{j,i})
 \big)+U^{n-\f{1}{2}}_{j,i}.
 \ee
It is clear that a fixed point $U_{j}^n$ with $TU_{j}^n=U_{j}^n$ is a solution to \eqref{eq:implicit}. 
Moreover, for any given functions $U^{n,(k)}_{j}$ and $V^{n,(k)}_{j}$ in $\ell^\infty (\Z^d)$, we obtain from the Lipschitz continuity of $f$ and the $\ell^\infty$ stability of the numerical flux $\tilde{f}$ (see Lemma \ref{lemma:lax}) that
\begin{align*}
|TU^{n,(k)}_{j}-TV^{n,(k)}_{j}|_0\le \bigg(\Delta t \big[\sum_{m\not =0}\kappa^{\a,n}_{h,m,i}+C
\big(1+\sum_{m\not =0}\b^{\a,n}_{h,m,i}\big)\big]+4d\theta\bigg)|U^{n,(k)}_{j}-V^{n,(k)}_{j}|_0.
\end{align*}
Since we need $h=o(\eps)$  in general  to achieve consistency of our scheme (see Lemma \ref{lemma:consistency_op}), it suffices to require $\f{\Delta t}{\eps^2}<1$ and $4d\theta<1$ to ensure $T$ is a contraction mapping on $\ell^\infty (\Z^d)$. This establishes the well-posedness of \eqref{eq:implicit} and enables us to solve the nonlinear equation \eqref{eq:implicit} through Picard iterations by setting 
$$
U^{n,(0)}_{j}=U^{n-\f{1}{2}}_{j},\qquad U^{n,(k+1)}_{j}=TU^{n,(k)}_{j},\q k\ge 0.
$$

We emphasize that the criterion $\f{\Delta t}{\eps^2}<1$ is a sufficient condition in the worst case, but is often far from computationally optimal
since we have used no information about the exact behavior of the singular measure $\nu$ around zero (see Remark \ref{rmk:bdd for coeff} for details). For typical L\'{e}vy measures from finance \cite{cont2005}, we only need  $\Delta t=O(\eps)$
(such as in the variance gamma case in our tests) or even $\Delta t$ independent of $\eps$ (for instance, for a Gaussian density) to guarantee $T$ is a contraction mapping.

Moreover, since in practice we can evaluate the discrete nonlocal operators $K^{\a,1}_{\eps,h}U^{n}_{j}$ and $B^{\a}_{\eps,h}U^{n}_{j}$ at all grid points in $O(N\log N)$ operations using a FFT (see e.g. \cite{{dHalluin2004}}), and in each iteration a sparse linear system is solved (in one dimension it is tridiagonal), the total complexity of the implicit scheme is still close to linear.

It is also worth pointing out that
even if we chose
explicit approximations for the nonlocal operators, due to the nonlinearity of $f$,  one still has to perform  iterations to solve  the resulting nonlinear equations. Because we only assume Lipschitz continuity but no higher regularity of $f$, we adopt the Picard iteration to solve for $U^{n}$.

\section{Convergence analysis}\l{sec:convergence}
In this section, we  establish the convergence of the numerical approximations in Section \ref{sec:scheme} to the viscosity solution of the HJBVI \eqref{eq:hjbvi}. 

We start by outlining the convergence analysis of the truncation of singular measures. 
It is not difficult to see that \eqref{eq:hjbvi_trun} is a consistent approximation of \eqref{eq:hjbvi} in the viscosity sense, such that the comparison principle of \eqref{eq:hjbvi} enables us to conclude that the solution of \eqref{eq:hjbvi_trun} converges to that of \eqref{eq:hjbvi} on compact sets as $\eps\to 0$.
In Appendix \ref{sec:jumps}, we provide an alternative proof by identifying the viscosity solution of \eqref{eq:hjbvi_trun} as the value function of a modified  control problem, and establish its convergence to the original value function \eqref{eq:nonlinear_control} using a probabilistic argument.

The remainder of this section thus focuses on the convergence analysis for the control discretization, the  approximation of  HJBVIs with switching systems, and the discrete approximations of the switching systems.

To simplify the notation, we will occasionally drop the terms  $\{K^\a u\}_{\a\in\bA}$ and $\{B^\a u\}_{\a\in\bA}$ in \eqref{eq:hjbvi_trun}, \eqref{eq:hjbvi_dc} and \eqref{eq:hjbvi_s}, and simply denote hem by $F(\x,u,Du,D^2u)=0$ in the sequel.

%
%
%
%

\subsection{Approximation by finite control sets}
In this section, we shall study approximations of HJBVI \eqref{eq:hjbvi} with a finite control set.  The following consistency result will be essential for our convergence analysis.
The proof is similar to \cite{reisinger2016} (who consider the case without nonlocal term, obstacle, and nonlinear driver) and included in Appendix \ref{app:cont-disc} for the convenience of the reader.

\begin{Lemma}\l{lem:consistency_control}
Under Assumption \ref{assum:coeff}, for any $\x\in \cQ_T$ and test function $\phi\in C^{1,2}(\bar{\cQ}_T)$  there exist functions 
$\omega_1(\x,\delta)$ and $\omega_2(\xi)$ such that
$\omega_1(\x,\delta)\to 0$ as $\delta\to 0$, $\omega_2(\xi)\to 0$ as $\xi\to 0$, and 
\bb
|F^\eps(\x,\phi(\x),D\phi(\x),D^2\phi(\x))-F^{\eps,\delta}(\x,\phi(\x)+\xi,D\phi(\x),D^2\phi(\x))|\le \omega_1(\x,\delta)+\omega_2(\xi),
\ee
where $\omega_1(\x,\delta)$ is locally Lipschitz continuous in $\x$ uniformly in $\delta$.
\end{Lemma}

Now we are ready to conclude the convergence of our approximation with finite control sets.
\begin{Theorem}
\label{theo:disc_cont}
Under Assumption \ref{assum:coeff}, let $u^\eps$ and $u^{\eps,\del}$ be the unique viscosity solution to \eqref{eq:hjbvi_trun} and \eqref{eq:hjbvi_dc}, respectively. Then
for fixed $\eps$ we have $u^{\eps,\del} \to u^\eps$ uniformly on compact sets as $\del\to 0$.
\end{Theorem}
The proof is a straightforward extension of the arguments in \cite{reisinger2016} and is hence omitted.

\subsection{Approximation by switching systems}\l{sec:conv_switching}

In this section, we study the approximation of \eqref{eq:hjbvi_dc} by switching systems. We adopt the following standard definition of a viscosity solution to switching systems of the form \eqref{eq:hjbvi_s} (see \cite{barles1997,biswas2010,reisinger2016} and references therein).

\begin{Definition}[Viscosity solution of  switching system]
A $\R^J$-valued  upper (resp.\ lower) semicontinuous function $U$ is said to be a viscosity subsolution (resp.\ supersolution) of \eqref{eq:hjbvi_s} if and only if 
for any point $\x_0$ and for any $\phi\in C^{1,2}(\bar{\cQ}_T)$ such that 
$U_j-\phi$ attains its global maximum (resp.\ minimum) at $\x_0$, one has
\begin{align*}
&F_{j*}^{\eps,\delta,c}(\x_0,U_j(\x_0), D\phi(\x_0), D^2\phi(\x_0),\{K_\eps^\a \phi(\x_0)\}_{\a\in\bA_\delta},\{B^\a_\eps \phi(\x_0)\}_{\a\in\bA_\delta},\{U_k(\x_0)\}_{k\not =j})\le 0\\
\big(resp. \q &F_j^{\eps,\delta,c\,*}(\x_0,U_j(\x_0), D\phi(\x_0), D^2\phi(\x_0),\{K_\eps^\a \phi(\x_0)\}_{\a\in\bA_\delta},\{B^\a_\eps \phi(\x_0)\}_{\a\in\bA_\delta}, \{U_k(\x_0)\}_{k\not =j})\ge 0\big).
\end{align*}
A continuous function is a viscosity solution of the  HJBVI \eqref{eq:hjbvi_s} if it is both a a viscosity sub- and supersolution.
\end{Definition}

Note that in the definition of the viscosity solution of $F_j$, the test function only replaces $U_j$ in the integrals and derivatives, while leaving the terms $\{U_k\}_{k\not=j}$ unchanged.

Now we present the comparison principle for bounded semicontinuous viscosity solutions of  \eqref{eq:hjbvi_s}, which not only implies the uniqueness of bounded viscosity solutions of \eqref{eq:hjbvi_s}, but is also essential for our convergence analysis. The proof will be given in Appendix \ref{sec:comparison}.
\begin{Theorem}\l{thm:comparision_s}
Let $U{=(U_1,U_2,...,U_J)}$ and $V {=(V_1,V_2,...,V_J)}$ be bounded viscosity sub- and supersolutions, respectively, of \eqref{eq:hjbvi_s} with $U(0,\cdot)\le V(0,\cdot)$. Then it holds under Assumption \ref{assum:coeff} that $U_j(\x)\le V_j(\x)$ for all $j=1,\ldots, J$.
\end{Theorem}

The following theorem demonstrates the convergence of the switching system to the finite control HJBVI \eqref{eq:hjbvi_dc} as the switching cost goes to $0$.
Convergence with order $1/3$ is proved in \cite{biswas2010} by a different technique, for nonlocal Bellman equations without obstacles and nonlinear source terms.

We momentarily assume the switching system \eqref{eq:hjbvi_s} to admit a viscosity solution bounded independently of the (small enough) switching cost $c$. We give a constructive proof of existence through our numerical schemes in Section \ref{sec:approx_switching}.
\begin{Theorem}\l{thm:conv_s}
Under Assumption \ref{assum:coeff}, 
let $U^{\eps,\delta, c}=(U^{\eps,\delta, c}_1,\ldots, U^{\eps,\delta, c}_J)$ and $u^{\eps,\delta}$ be the viscosity solution of \eqref{eq:hjbvi_s} and \eqref{eq:hjbvi_dc}, respectively.  Then for fixed $\eps, \delta>0$, we have for each $j=1,\ldots, J$ that
$
U^{\eps,\delta, c}_j\to u^{\eps,\delta}
$
uniformly on compact sets as $c\to 0$.
\end{Theorem}

\begin{proof}
Since  $\eps$ and  $\delta$ are fixed for our  analysis, we shall omit the dependence on  $\eps$ and  $\delta$, and simply denote by $U^{c}$  the  solution of \eqref{eq:hjbvi_s}.
Consider a sequence of switching costs $c_m\to 0$ as $m\to \infty$, and the corresponding viscosity solution $U^{c_m}=(U^{c_m}_1,\ldots, U^{c_m}_J)$. We shall first prove by contradiction that 
\bb\l{eq:cM}
U^{c_m}_j(\x)\ge \cM_j U^{c_m},\q \x\in \cQ_T, \; j=1,\ldots, J.
\ee
Suppose the statement is false, then  there would exist $k\not= j$ and $\x_0\in \cQ_T$ such that $U^{c_m}_j(\x_0)< U^{c_m}_k(\x_0)-c_m$. We then obtain from the continuity of $U_j^{c_m}$ and $U^{c_m}_k$ that there exists  a nonempty open ball $B$ around $\x_0$ such that
$$
U^{c_m}_j(\x)< U^{c_m}_k(\x)-c_m, \q \x\in B.
$$
On the other hand, 
there exists a $C^2$ function $\phi$ such that $U^{c_m}_j-\phi$ attains its minimum at some point in $B$, say $\x_1$. Hence we deduce from the fact that $U^{c_m}_j$ is a supersolution that
$$
U^{c_m}_j(\x_1)\ge \cM_j U^{c_m}(\x_1)\ge U^{c_m}_k(\x_1)-c_m,
$$
which leads to a contradiction. 

We now introduce the following functions through a relaxed limit: for $j=1,\ldots, J$,
\begin{align}
\overline{U}_j(x)=\lim_{r\to \infty}\sup_{m>r}\sup_{|\y-\x|<1/r}U^{c_m}_j(\y),\q
\ul{U}_j(x)=\lim_{r\to \infty}\inf_{m>r}\inf_{|\y-\x|<1/r}U^{c_m}_j(\y).
\end{align}
It is not hard to check $\overline{U}_1=\ldots=\overline{U}_J\equiv\overline{U}$ and $\ul{U}_1=\ldots=\ul{U}_J\equiv\ul{U}$. In fact, for any given $j,k\in \{1,\ldots, J\}$, $j\not =k$, $\x\in \cQ_T$, and $m,r\in \N$,  we obtain from \eqref{eq:cM} that
$U^{c_m}_j(\y)\ge U^{c_m}_k(\y)-c_m$ for $\y\in \cQ_T$, and hence 
$$
\sup_{m>r}\sup_{|\y-\x|<1/r}U^{c_m}_j(\y)\ge \sup_{m>r}\sup_{|\y-\x|<1/r}U^{c_m}_k(\y)-\sup_{m>r}c_m.
$$
Letting $r\to\infty$ leads to the fact that $\overline{U}_j\ge \overline{U}_k$ for all $j\not= k$. The statement for $\{\ul{U}_j\}$ can be shown similarly.

Since it is clear that $\overline{U}$ and $\ul{U}$ is bounded upper and lower semicontinuous, respectively, we now aim to show $\overline{U}$ and $\ul{U}$ is respectively a sub- and supersolution of \eqref{eq:hjbvi_dc}. Then  the strong comparision principle gives us $\overline{U}\le \ul{U}$, which implies $U=\overline{U}= \ul{U}$ is the unique viscosity solution of \eqref{eq:hjbvi_dc}. Uniform convergence on compact sets follows from a variation of Dini's theorem (See Remark 6.4 in \cite{crandall1992}).

We start by showing $\overline{U}$ is a subsolution of \eqref{eq:hjbvi_dc}. 
Let $\phi\in C^{1,2}$ and $\overline{U}-\phi$ have a strict global maximum at $\hat{\x}_0\in \bar{\cQ}_T$, then there will be a sequence  $c_m\to 0$ such that for each $j\in \{1,\ldots, J\}$, we have $\hat{\x}^j_m\to \hat{\x}_0$, $U_j^{c_m}(\hat{\x}^j_m)\to\overline{U}(\hat{\x}_0)$, and $U_j^{c_m}-\phi$ attains a global maximum at $\hat{\x}^j_m$. 
Since $U_j^{c_m}$ is a subsolution of \eqref{eq:hjbvi_s} with $c_m$, 
if we have $\hat{\x}_0\in \{0\}\t \R^d$, $U_j^{c_m}(\hat{\x}^j_m)\le g(\hat{\x}^j_m)$ for infinitly many $m$ and a fixed $j$, then it is clear that $\overline{U}(\hat{\x}_0)\le g(\hat{\x}_0)$. Therefore,
without loss of generality, we assume for all $m$ and $j$ that
\begin{align}
\min\bigg[U_j^{c_m}(\hat{\x}^j_m)-\zeta(\hat{\x}^j_m),\; \min\big(&\phi_{t}(\hat{\x}^j_m)-L_\eps^{\a_j}  \phi(\hat{\x}^j_m)-f^{\a_j}(\hat{\x}^j_m,U^{c_m}_{j}(\hat{\x}^j_m),\tilde{\sigma}^{\a_j}\cdot D\phi(\hat{\x}^j_m),B_\eps^{\a_j} \phi(\hat{\x}^j_m)); \nb\\
& U^{c_m}_{j}(\hat{\x}^j_m)-\cM_jU^{c_m}(\hat{\x}^j_m)\big)\bigg]\le 0. \l{eq:subsoln_s1}
\end{align}

We have two cases. If there exists $j\in\{1,\ldots, J\}$ and a subsequence of $c_m$ such that $U_j^{c_m}(\hat{\x}^j_m)-\zeta(\hat{\x}^j_m)\le 0$, then by passing to the limit $m\to \infty$, we have $\overline{U}(\hat{\x}_0)-\zeta(\hat{\x}_0)\le 0$. Otherwise, by passing to subsequence, without loss of generality we can assume $U_j^{c_m}(\hat{\x}^j_m)-\zeta(\hat{\x}^j_m)>0$ holds for all $j$ and $m$. Then for each $m\in \N$, we can choose $j_m\in \{1,\ldots, J\}$ and $\hat{\x}^{j_m}_m$ such that
$$
(U_{j_m}^{c_m}-\phi)(\hat{\x}^{j_m}_m)=\max_{j=1,\ldots, J}(U_{j}^{c_m}-\phi)(\hat{\x}^{j}_m)=\max_{j=1,\ldots, J}\max_{\x}(U_{j}^{c_m}-\phi)(\x), \l{eq:jm}
$$
and deduce from \eqref{eq:subsoln_s1} that 
\begin{align}
\min\big(\phi_{t}(\hat{\x}^{j_m}_m)-L_\eps^{\a_{j_m}}  \phi(\hat{\x}_m)-&f^{\a_{j_m}}(\hat{\x}^{j_m}_m,U^{c_m}_{{j_m}}(\hat{\x}^{j_m}_m),\tilde{\sigma}^{\a_{j_m}}\cdot D\phi(\hat{\x}^{j_m}_m),B_\eps^{\a_{j_m}}\phi(\hat{\x}^{j_m}_m)); \nb\\
 &U^{c_m}_{{j_m}}(\hat{\x}^{j_m}_m)-\cM_{j_m}U^{c_m}(\hat{\x}^{j_m}_m)\big)\le 0. \l{eq:subsoln_s2}
\end{align}

Our choice of $j_m$ implies $(U_{j_m}^{c_m}-\phi)(\hat{\x}^{j_m}_m)\ge (U_{k}^{c_m}-\phi)(\hat{\x}^{j_m}_m)$ for all $k\not= j_m$, and thus  $U^{c_m}_{{j_m}}(\hat{\x}^{j_m}_m)>\cM_{j_m}U^{c_m}(\hat{\x}^{j_m}_m)$. Consequently we obtain from \eqref{eq:subsoln_s2} that 
$$
\phi_{t}(\hat{\x}^{j_m}_m)-L_\eps^{\a_{j_m}}  \phi(\hat{\x}_m)-f^{\a_{j_m}}(\hat{\x}^{j_m}_m,U^{c_m}_{{j_m}}(\hat{\x}^{j_m}_m),\tilde{\sigma}^{\a_{j_m}}\cdot D\phi(\hat{\x}^{j_m}_m),B_\eps^{\a_{j_m}} \phi(\hat{\x}^{j_m}_m))\le 0.
$$
Since we only have finite many choices of $j_m$, by passing to a further subsequence if necessary, we can assume that $j_m\to j_0$, then letting $m\to \infty$ and using the continuity of the equation, we have 
$$
\phi_{t}(\hat{\x}_0)-L_\eps^{\a_{j_0}}  \phi(\hat{\x}_0)-f^{\a_{j_0}}(\hat{\x}_0,\overline{U}(\hat{\x}_0),\tilde{\sigma}^{\a_{j_0}}\cdot D\phi(\hat{\x}_0),B_\eps^{\a_{j_0}} \phi(\hat{\x}_0))\le 0.
$$
Since $\a_{j_0}\in \bA_\delta$ is an admissible control, we obtain 
$$
\min_{\a\in\bA_\delta}\big\{\phi_{t}(\hat{\x}_0)-L_\eps^{\a}  \phi(\hat{\x}_0)-f^{\a_{j}}(\hat{\x}_0,\overline{U}(\hat{\x}_0),\tilde{\sigma}^{\a}\cdot D\phi(\hat{\x}_0),B_\eps^\a \phi(\hat{\x}_0))\big\}\le 0,
$$
and conclude that $\overline{U}$ is a subsolution of \eqref{eq:hjbvi_s}.

We now proceed to show $\ul{U}$ is a supersolution. 
If $\phi\in C^{1,2}$ and $\overline{U}-\phi$ has a strict global mimimum at $\hat{\x}_0\in \bar{\cQ}_T$, then for any given   $j\in \{1,\ldots, J\}$, there will be sequences  $c_m\to 0$, $\hat{\x}_m\to \hat{\x}_0$, $U_j^{c_m}(\hat{\x}_m)\to\ul{U}(\hat{\x}_0)$, and $U_j^{c_m}-\phi$ attains a global mimimum at $\hat{\x}_m$. Using the fact that $U_j^{c_m}$ is asupersolution to \eqref{eq:hjbvi_s}, we have (by ignoring the term $U^{c_m}_{j}(\hat{\x}^j_m)-\cM_jU^{c_m}(\hat{\x}^j_m)$):
\begin{align}
\min\bigg[U_j^{c_m}(\hat{\x}_m)-\zeta(\hat{\x}_m),\; \phi_{t}(\hat{\x}_m)-L_\eps^{\a_j}  \phi(\hat{\x}_m)-f^{\a_j}(\hat{\x}_m,U^{c_m}_{j}(\hat{\x}_m),\tilde{\sigma}^{\a_j}\cdot D\phi(\hat{\x}_m),B_\eps^{\a_j} \phi(\hat{\x}_m))\bigg]\ge 0, \nb
\end{align}
then passing $m\to \infty$ enables us to conclude for any $j\in \{1,\ldots, J\}$,
\begin{align}
\min\bigg[\ul{U}(\hat{\x}_0)-\zeta(\hat{\x}_0),\; \phi_{t}(\hat{\x}_0)-L_\eps^{\a_j}  \phi(\hat{\x}_0)-f^{\a_j}(\hat{\x}_0,\ul{U}(\hat{\x}_0),\tilde{\sigma}^{\a_j}\cdot D\phi(\hat{\x}_0),B_\eps^{\a_j} \phi(\hat{\x}_0))\bigg]\ge 0, \nb
\end{align}
which completes our proof.
\end{proof}

\subsection{General discrete approximation to the switching system}\l{sec:approx_switching}
In this section, we  establish the convergence of the piecewise constant policy approximation of \eqref{eq:scheme} to the solution of the switching system \eqref{eq:hjbvi_s}. We will first summarize all the required conditions to guarantee the convergence, and perform the analysis under these assumptions. Then we will demonstrate in Section \ref{sec:implicit} that these conditions are in fact satisfied by the numerical scheme \eqref{eq:implicit} proposed  in Section \ref{sec:scheme} .

We assume  the  scheme \eqref{eq:scheme} satisfies the following  conditions  introduced in \cite{reisinger2016}:
\begin{Condition}\l{assum:scheme}
\begin{enumerate}[(1)]
\item  (Positive interpolation.) Let $\tilde{U}^n_{k,i(j)}$ be the interpolant of the $k$-th grid onto the $i$-th point $\x_{j,i}^n$ of the $j$-th grid, and $N^k(j,i,n)$ be the
neighbours\footnote{``Neighbours'' can be any a set of indices $a$ such that $\x^n_{k,a} \rightarrow \x^n_{j,i}$.}
to the point $\x_{j,i}^n$ on the $k$-th grid $\Omega_{k,h}$. Then there exist  weights $\{\omega^n_{k,i(j),a}\}_{a\in N^k(j,i,n)}$ satisfying $\omega^n_{k,i(j),a}\ge 0$ and $\sum_{a\in N^k(j,i,n)}\omega^n_{k,i(j),a}=1$, such that we can write
\bb\l{eq:interpolation}
\tilde{U}^n_{k,i(j)}=\sum_{a\in N^k(j,i,n)}\omega^n_{k,i(j),a} U^n_{k,a}.
\ee
\item (Weak monotonicity.) The scheme \eqref{eq:scheme} is monotone with respect to $U^n_{j,i}$ and  $\tilde{U}^{n}_{k,i(j)}$, i.e., if 
$$
V^n_{j,i}\ge U^n_{j,i}, \q \forall (i,j,n); \q  \tilde{V}^{n}_{k,i(j)}\ge \tilde{U}^{n}_{k,i(j)}, \q \fa (i,k,n),
$$
then we have
\bb
G_j(\x^n_{j,i},h,U^{n+1}_{j,i}, \{V^{b+1}_{j,a}\}_{\textnormal{$(a,b)\!\not =\!(i,n)$}}, \{
\tilde{V}^n_k\}_{k\not =j})
\le G_j(\x^n_{j,i},h,U^{n+1}_{j,i}, \{U^{b+1}_{j,a}\}_{\textnormal{$(a,b)\!\not =\!(i,n)$}}, \{
\tilde{U}^n_k\}_{k\not =j}).
\ee
\item ($\ell^\infty$ stability.) The solution $U^{n+1}_{j,i}$ of the scheme \eqref{eq:scheme} exists and is bounded uniformly in $h$ and $c$.
\item (Consistency.) Let $\eps,\delta, c$ be fixed. For any test functions $\phi_j\in C^{1,2}(\bar{\cQ}_T)$ and continuous $\varphi_k$, 
there exist function $\omega_1(h)$ and $\omega_2(\xi)$, {possibly depending on $\eps$},  such that $\omega_1(h)\to 0$  as $h\to 0$, $\omega_2(\xi)\to 0$  as $\xi\to 0$, and 
\begin{align}\l{eq:consistent_scheme}
\begin{split}
|&G_j(\x^{n+1}_{j,i},h,\phi^{n+1}_{j,i}+\xi, \{\phi^{b+1}_{j,a}\}_{\textnormal{$(a,b)\!\not =\!(i,n)$}}+\xi, \{\tilde{\varphi}^n_k\}_{k\not =j})\\
&-F^{\eps,\delta, c}_j(\x^{n+1}_{j,i},\phi_j(\x^{n+1}_{j,i}), D\phi_j(\x^{n+1}_{j,i}), D^2\phi_j(\x^{n+1}_{j,i}),\{\tilde{\varphi}_k(\x^{n}_{j,i})\}_{k\not =j})|\le \omega_1(h)+\omega_2(\xi).
\end{split}
\end{align}
\end{enumerate}
\end{Condition}

\begin{Remark}
As pointed out in \cite{reisinger2016}, Condition \ref{assum:scheme} (1)-(2) are weaker than the standard condition that  the scheme is monotone in $U^n_{k,\a}$.
By  only requiring that the interpolation has positive coefficients and that the numerical scheme is monotone in the interpolant $\tilde{U}^n_{k,\a}$, we are allowing the usage of high order nonlinear interpolations among different grids (e.g., the monotonicity preserving interpolations in \cite{fritsch1980}).

Also note the contrast to the linear interpolant \eqref{eq:lininterp} used in \eqref{eq:K_mono} and \eqref{eq:B_mono} for the construction of a monotone approximation to the integral operators.
\end{Remark}

We now present the convergence of the discrete approximation to the switching system. 
\begin{Theorem}\l{thm:conv_numerical}
Under Assumptions \ref{assum:coeff},  the solution to any scheme of the form \eqref{eq:scheme} satisfying Condition \ref{assum:scheme} converges to the viscosity solution of \eqref{eq:hjbvi_s} uniformly on bounded domains.
\end{Theorem}

The proof is essentially the same as that in \cite{reisinger2016} and is omitted. We remark that in the proof, we construct the solution of the switching system directly from the numerical solutions. Since the solution of the scheme \eqref{eq:scheme} is uniformly bounded, Theorems \ref{thm:comparision_s} and \ref{thm:conv_numerical}  immediately give the existence and uniqueness of a bounded viscosity solution to the switching system \eqref{eq:hjbvi_s}.
\begin{Corollary}
Under  Assumption \ref{assum:coeff} and the existence of a scheme satisfying Condition \ref{assum:scheme}, the switching system \eqref{eq:hjbvi_s} admits a unique viscosity solution bounded uniformly in $c$.
\end{Corollary}

\subsection{A specific implicit scheme for the switching system}\l{sec:implicit}
In this section, we analyze the implicit scheme \eqref{eq:implicit} and demonstrate that it satisfies Condition \ref{assum:scheme}, which subsequently implies its convergence to the switching system.

The following estimates are essential for our consistency and  stability analysis.
\begin{Lemma}\l{lemma:consistency_op}
Under Assumption \ref{assum:coeff}, 
there exists $C$ independent of $h,k,\eps, \delta$ such that
for any test functions $\phi_j\in C^{1,2}(\bar{\cQ}_T)$ 
and $\varepsilon<1$
that
\begin{align*}
&|A^{\a}_{\eps,h,k}\phi^{n+1}_{j,i}+K^{\a,1}_{\eps,h}\phi^{n+1}_{j,i}
-A_\eps^{\a} \phi_j(\x_{j,i}^{n+1})-K_\eps^{\a} \phi_j(\x_{j,i}^{n+1})|\le C \left(\f{h^2}{k^2}+\f{h^2}{\eps^2}+\om(\x^{n+1}_{j,i},k)\right),\\
&|B^{\a}_{\eps,h}\phi^{n+1}_{j,i}-B_\eps^{\a}\phi(\x^{n+1}_{j,i})|\le C\f{h^2}{\eps}.
\end{align*}
for some  $\om(\x^{n+1}_{j,i},k)$ such that $\om(\cdot, k)\to 0$ as $k\to 0$ uniformly on compact neighbourhoods of $\x^{n+1}_{j,i}$.
\end{Lemma}
\color{black}
\begin{proof}
We first derive the estimate for $B^{\a}_{\eps,h}\phi^{n+1}_{j,i}$. It follows from $|\eta^\a|\le C$ and the definitions of $B^{\a}_{\eps,h}\phi$ and $B_\eps^{\a}\phi$ that
\begin{align*}
&|B^{\a}_{\eps,h}\phi^{n+1}_{j,i}-B_\eps^{\a}\phi(\x^{n+1}_{j,i})|\\
\le &
\int_{|e|\ge \eps}|\cI_h[\phi(t^{n+1},x_{j,i}+\cdot)](\eta^\a(x_{j,i},e))-\phi(t^{n+1},x_{j,i}+\eta^\a(x_{j,i},e))|
\gamma(x_{j,i},e)\,\nu(de)\\
\le &Ch^2|D^2\phi|_{B(\x^{n+1}_{j,i},C)}\int_{|e|\ge \eps}(1\wedge |e|)\,\nu(de)\le  C\f{h^2}{\eps},
\end{align*}
where we have used the fact that $|\cI_h[\phi]-\phi|_{B(\x^{n+1}_{j,i},C)}\le C|D^2\phi|_{B(\x^{n+1}_{j,i},C)}h^2$.
Similar arguments give us that $|K^{\a,1}_{\eps,h}\phi^{n+1}_{j,i}-K^{\a,1}_{\eps}\phi(\x^{n+1}_{j,i})|\le Ch^2|D^2\phi|_{B(\x^{n+1}_{j,i},C)}\int_{|e|\ge \eps}\,\nu(de)\le C\f{h^2}{\eps^2}$.

We then infer from Taylor's theorem with an integral remainder  that the truncation errors of the local terms can be bounded by
$$|A^{\a}_{\eps,h,k}\phi^{n+1}_{j,i}-A^{\a}_{\eps}\phi(\x^{n+1}_{j,i})-b^\a_{\eps}(x_{j,i})\cdot D\phi(\x^{n+1}_{j,i})|\le C|D^2\phi|_{B(\x^{n+1}_{j,i},C)}\f{h^2}{k^2}+\om(\x^{n+1}_{j,i},k)$$
for some function $\om(\x^{n+1}_{j,i},k)$ such that $\om(\cdot, k)\to 0$ as $k\to 0$ uniformly on compact neighbourhoods of $\x^{n+1}_{j,i}$,
\color{black}
which 
 enables us to deduce that
\begin{align*}
&|A^{\a}_{\eps,h,k}\phi^{n+1}_{j,i}+K^{\a,1}_{\eps,h}\phi^{n+1}_{j,i}
-A^{\a}_\eps \phi_j(\x_{j,i}^{n+1})-K^{\a} _\eps\phi_j(\x_{j,i}^{n+1})|\\
\le &|A^{\a}_{\eps,h,k}\phi^{n+1}_{j,i}-A^{\a}_{\eps}\phi(\x^{n+1}_{j,i})-b^\a_{\eps}(x_{j,i})\cdot D\phi(\x^{n+1}_{j,i})|+|K^{\a,1}_{\eps,h}\phi^{n+1}_{j,i}-K^{\a,1}_{\eps}\phi(\x^{n+1}_{j,i})|\\
\le &C\left(\f{h^2}{k^2}+\f{h^2}{\eps^2}+\om(\x^{n+1}_{j,i},k)\right).
\end{align*}
\end{proof}

\begin{Lemma}\l{lemma:sum_coeff}
Under Assumption \ref{assum:coeff}
there exists $C$ independent of $h,k,\eps, \delta$ such that
for all $\varepsilon<1$
$$
\sum_{m\not =0}\kappa^{\a,n}_{h,m,i}\le \f{C}{h\eps}\wedge \f{1}{\eps^2},\q
\sum_{m\not =0}\b^{\a,n}_{h,m,i}\le \f{C}{h}\wedge \f{1}{\eps}, \q
\sum_{m\in \Z^d}d^{\a,n}_{h,k,m,i} \le  \f{C}{k^2}, 
$$
where $\kappa^{\a,n}_{h,m,i}$, $\b^{\a,n}_{h,m,i}$, and $d^{\a,n}_{h,k,m,i}$ are defined in \eqref{eq:k&b_coeff} and \eqref{eq:d_coeff}, respectively.
\end{Lemma}
\begin{proof}
We shall only prove the estimate for $\kappa^{\a,n}_{h,m,i}$, since the estimate for $\b^{\a,n}_{h,m,i}$ follows from a similar argument, and the estimate for $d^{\a,n}_{h,k,m,i}$ follows directly from the fact that $\sum_m \om_m=1$.

The definition of $\kappa^{\a,n}_{h,m,i}$ and the integrability property \eqref{eq:nu_int} of $\nu$  imply that
\begin{align*}
\sum_{m\not=0}\kappa^{\a,n}_{h,m,i}&=\sum_{m\not=0}\int_{|e|>\eps} \om_m(\eta^\a(x_i,e);h)\,\nu(de)=\sum_{m\not=0}\int_{|e|>\eps} \om_m(\eta^\a(x_i,e);h)1_{\{\eta^\a(x_i,e)\in \textnormal{supp}\,\om_m\}}\,\nu(de)\\
&=\sum_{m\not=0}\int_{|e|>\eps} \big(\om_m(\eta^\a(x_i,e);h)-\om_m(0;h)\big)1_{\{\eta^\a(x_i,e)\in \textnormal{supp}\,\om_m\}}\,\nu(de)\\
&\le \int_{|e|>\eps}\sum_{m\not=0} |D\om_m|_0|\eta^\a(x_i,e)|1_{\{\eta^\a(x_i,e)\in \textnormal{supp}\,\om_m\}}\,\nu(de)\le \f{C}{h}\int_{|e|>\eps}(1\wedge |e|)\,\nu(de)\\
&\le \f{C}{h}\int_{|e|>\eps}\f{1\wedge |e|}{\eps}(1\wedge |e|)\,\nu(de)=\f{C}{h\eps}\int_{|e|>\eps}(1\wedge |e|^2)\,\nu(de)\le \f{C}{h\eps}.
\end{align*}
Alternatively, it follows directly from the identity $\sum_{m\in \Z^d}\om_m(\cdot;h)\equiv 1$ that
$$
\sum_{m\not=0}\kappa^{\a,n}_{h,m,i}=\sum_{m\not=0}\int_{|e|>\eps} \om_m(\eta^\a(x_i,e);h)\,\nu(de)\le \int_{|e|>\eps}\,\nu(de)\le \f{1}{\eps^2}\int_{|e|>\eps}(1\wedge |e|^2)\,\nu(de),
$$
which leads us to the desired estimates.
\end{proof}

\begin{Remark}\l{rmk:bdd for coeff}
Since we have not used any information on the exact behavior of the nonsingular measure $\nu$ around zero,
the estimates for the nonlocal terms in Lemma \ref{lemma:consistency_op} and \ref{lemma:sum_coeff} are not optimal for many specific cases.
 If one can estimate  upper bounds of the density of the L\'{e}vy measure, or equivalently estimate the (pseudo-differential) orders of the nonlocal operators $K^\a$ and $B^\a$,  more precise results for the truncation error of the singular measure can be deduced (\cite{biswas2017}).
\end{Remark}
The next lemma presents some important properties of the Lax-Friedrichs numerical flux for Lipschitz continuous Hamiltonian, which are crucial for our subsequent analysis.  We refer readers to \cite{crandall1984} for a proof of these statements. Then the following hold:
\begin{Lemma}\l{lemma:lax}
Let 
$\tilde{f}$ as in \eqref{eq:numerical_flux} and $(\x^{n}_{j,i}, u, k)\in \Omega_{j,h}\t \R\t \R$, and suppose  Assumption \ref{assum:coeff} and the condition $\theta>C\lambda$ hold, where $C$ is the Lipschitz constant of the Hamiltonian $\bar{f}$.
\begin{enumerate}[(1)]
\item (Consistency.) 
For any test functions $\phi\in C^{1,2}([0,T]\t \R^d)$, we have 
$$
|\tilde{f}^{\a}(\x^{n}_{j,i},u,  \Delta \phi^{n}_{j,i},k)-\bar{f}^{\a}(\x_{{j,i}}^{{n}},u, D\phi(\x^{n}_{j,i}),k)|\le Ch^2/\Delta t.
$$
\item {(Monotonicity.)} 
If $V^n_{j,i}\ge U^n_{j,i}$, for all $i,j,n$, then we have
$$
\Delta t\tilde{f}^{\a}(\x^{n}_{j,i},u,  \Delta V^{n}_{j,i},k)+2d\theta V_{j,i}^n\ge \Delta t\tilde{f}^{\a}(\x^{n}_{j,i},u,  \Delta U^{n}_{j,i},k)+2d\theta U_{j,i}^n.
$$
\item { (Stability.)} 
For any bounded functions $U$ and $V$, we have
$$|(\Delta t\tilde{f}^{\a}(\x^{n}_{j,i},u,  \Delta V^{n}_{j,i},k)+2d\theta V_{j,i}^n)
-( \Delta t\tilde{f}^{\a}(\x^{n}_{j,i},u,  \Delta U^{n}_{j,i},k)+2d\theta U_{j,i}^n)|
\le 2d\theta|U-V|_0.$$
\end{enumerate}

\end{Lemma}

\begin{Proposition} Suppose Assumption \ref{assum:coeff}, the positive interpolation property in Condition \ref{assum:scheme} and the  condition $\theta>C\lambda$ hold. Then we have the following:
\begin{enumerate}[(1)]
\item There exists a unique bounded solution $U^{n}$ of the scheme \eqref{eq:implicit}.
\item The scheme is  $\ell^\infty$ stable and weakly monotone.
It is consistent with the switching system \eqref{eq:hjbvi_s} provided $h^2/\Delta t\rightarrow 0$ and $h/k\rightarrow 0$ as
$h,k,\Delta t \rightarrow 0$ ($\eps$ is fixed here). 
\end{enumerate}
\end{Proposition}

\begin{proof}
We start to establish the existence and uniqueness of a bounded solution of \eqref{eq:implicit} in (1) by an induction argument. It is clear the statement holds for $t^0=0$ since $U^0=g$ is bounded. Now we assume that $\{U_j^{n-1}\}_{j=1}^J$ are bounded functions on $h\Z^d$ and consider the time point $t^n$. The positive interpolation property implies the interpolation step among different grids does not increase the $\ell^\infty$ norm of the solution, and hence $U_j^{n-\f{1}{2}}$ is bounded for each $j=1,\ldots, J$.

For each $\rho>0$ and $j=1,\ldots, J$, we define the operator $\cP:U_j^n\to U_j^n$ by
$$
\cP U_{j,i}^n=U_{j,i}^n-\rho\cdot \textnormal{(left-hand side of \eqref{eq:implicit})}, \q  i\in \Z^d,
$$
with a given function $U^{n-\f{1}{2}}_j$.
By virtue of the fact that fixed points to the equation $\cP U_{j}^n=U_j^n$ are precisely the solutions to \eqref{eq:implicit}, it suffices to establish that for small enough $\rho$, the operator $\cP$ is a contraction on $\ell^\infty(\Z^d)$, i.e., the Banach space of bounded functions on  $h\Z^d$ employed with the sup-norm, which along with the contraction mapping theorem leads to the desired results.
{(Similar contraction operators have been introduced in \cite{biswas2010diff,debrabant2012} to demonstrate the well-posedness of their numerical schemes.)}

For any bounded functions $U_j^n$ and $V_j^n$, the definitions of $\cP$, $A^{\a}_{\eps,h,k}$ and $K^{\a,1}_{\eps,h}$ give that
\begin{align}
&\cP U_{j,i}^n-\cP V_{j,i}^n\nb\\
\le&(1-\rho)(U_{j,i}^n-V_{j,i}^n)+\rho\Delta t \bigg[\sum_{m\in \Z^d}d^{\a,n}_{h,k,m,i}[(U_{j,m}^n-V_{j,m}^n)-(U_{j,i}^n-V_{j,i}^n)]\nb\\
&+ \sum_{m\not=0}\kappa^{\a,n}_{h,m,i}[(U_{j,i+m}^n-V_{j,i+m}^n)-(U_{j,i}^n-V_{j,i}^n)]\nb\\
&+\tilde{f}^{\a}(\x^{n}_{j,i},U^{n}_{j,i},\Delta U^{n}_{j,i}, B^{\a}_{\eps,h}U^{n}_{j,i})
-\tilde{f}^{\a}(\x^{n}_{j,i},V^{n}_{j,i},\Delta V^{n}_{j,i},B^{\a}_{\eps,h}V^{n}_{j,i}) \bigg]\nb\\
\le &(1-\rho-\rho\Delta t \sum_{m\in \Z^d}d^{\a,n}_{h,k,m,i}-\rho\Delta t \sum_{m\not=0}\kappa^{\a,n}_{h,m,i})(U_{j,i}^n-V_{j,i}^n)+\rho\Delta t( \sum_{m\in \Z^d}d^{\a,n}_{h,k,m,i}+\sum_{m\not=0}\kappa^{\a,n}_{h,m,i})|U^{n}_{j}-V^{n}_{j}|_0\nb\\
&+\rho\Delta t \big(
\tilde{f}^{\a}(\x^{n}_{j,i},U^{n}_{j,i},\Delta U^{n}_{j,i}, B^{\a}_{\eps,h}U^{n}_{j,i})
-\tilde{f}^{\a}(\x^{n}_{j,i},V^{n}_{j,i},\Delta U^{n}_{j,i}, B^{\a}_{\eps,h}U^{n}_{j,i})
 \big)
\l{eq:ctrc_line1}\\
&+\rho\Delta t \big(
\tilde{f}^{\a}(\x^{n}_{j,i},V^{n}_{j,i},\Delta U^{n}_{j,i}, B^{\a}_{\eps,h}U^{n}_{j,i})
-\tilde{f}^{\a}(\x^{n}_{j,i},V^{n}_{j,i},\Delta U^{n}_{j,i}, B^{\a}_{\eps,h}V^{n}_{j,i})
 \big) \l{eq:ctrc_line2}\\
&+\rho\Delta t \big(
\tilde{f}^{\a}(\x^{n}_{j,i},V^{n}_{j,i},\Delta U^{n}_{j,i}, B^{\a}_{\eps,h}V^{n}_{j,i}) 
-\tilde{f}^{\a}(\x^{n}_{j,i},V^{n}_{j,i},\Delta V^{n}_{j,i},B^{\a}_{\eps,h}V^{n}_{j,i})\big).\l{eq:ctrc_line3}
%
\end{align}

It remains to estimate \eqref{eq:ctrc_line1}, \eqref{eq:ctrc_line2} and \eqref{eq:ctrc_line3}.
Lemma \ref{lemma:lax} (3) enables us to bound \eqref{eq:ctrc_line3} by
$
-\rho 2d\theta(U^{n}_{j,i}-V^{n}_{j,i})+\rho 2d\theta|U^{n}_{j}-V^{n}_{j}|_0.
$
We then derive upper bounds for \eqref{eq:ctrc_line1} and \eqref{eq:ctrc_line2} depending on whether $U^n_{j,i}-V^n_{j,i}$ or $B^{\a}_{\eps,h}U^{n}_{j,i}-B^{\a}_{\eps,h}V^{n}_{j,i}$ is positive. If $U^n_{j,i}-V^n_{j,i}>0$,  the monotonicity of $f$ in $y$ implies that $\eqref{eq:ctrc_line1}$ is bounded above by $-\rho\Delta tC(U^n_{j,i}-V^n_{j,i})$, while if $U^n_{j,i}-V^n_{j,i}<0$, the Lipschitz continuity  of $f$ in $y$ enables us to bound  $\eqref{eq:ctrc_line1}$ by $\rho\Delta tC|U^n_{j,i}-V^n_{j,i}|=-\rho\Delta tC(U^n_{j,i}-V^n_{j,i})$.

We then discuss the sign of $B^{\a}_{\eps,h}U^{n}_{j,i}-B^{\a}_{\eps,h}V^{n}_{j,i}$. Suppose  $B^{\a}_{\eps,h}U^{n}_{j,i}-B^{\a}_{\eps,h}V^{n}_{j,i}<0$, then we obtain from the monotonicity of $f$ in $k$  that $\eqref{eq:ctrc_line2}\le 0$. Consequently we obtain that
\begin{align}
\cP U_{j,i}^n-\cP V_{j,i}^n\nb\le &(1-\rho-\rho\Delta t \sum_{m\in \Z^d}d^{\a,n}_{h,k,m,i}-\rho\Delta t \sum_{m\not=0}\kappa^{\a,n}_{h,m,i}-\rho\Delta t C-\rho 2d\theta)(U_{j,i}^n-V_{j,i}^n)\\
&+\rho(\Delta t \sum_{m\in \Z^d}d^{\a,n}_{h,k,m,i}+\Delta t\sum_{m\not=0}\kappa^{\a,n}_{h,m,i}+ 2d\theta)|U^{n}_{j}-V^{n}_{j}|_0\nb\\
\le &(1-\rho-\rho\Delta t C)|U_{j}^n-V_{j}^n|_0, \l{eq:P}
\end{align}
provided that $1-\rho(1+2d\theta)-\rho\Delta t\big( \sum_{m\in \Z^d}d^{\a,n}_{h,k,m,i}+\sum_{m\not=0}\kappa^{\a,n}_{h,m,i}+ C)>0$, which is satisfied for small enough $\rho$.

On the other hand, if $B^{\a}_{\eps,h}U^{n}_{j,i}-B^{\a}_{\eps,h}V^{n}_{j,i}>0$, the Lipschitz continuity of $f$ in $k$ enables us to bound  $\eqref{eq:ctrc_line2}$ by $C (B^{\a}_{\eps,h}U^{n}_{j,i}-B^{\a}_{\eps,h}V^{n}_{j,i})$, which along with \eqref{eq:B_mono} implies again
 \eqref{eq:P} provided that the the following condition  is satisfied:
\bb\l{eq:stability_condition}
1-\rho(1+2d\theta)-\rho\Delta t\bigg( \sum_{m\in \Z^d}d^{\a,n}_{h,k,m,i}+\sum_{m\not=0}(\kappa^{\a,n}_{h,m,i}+\b^{\a,n}_{h,m,i})+ C\bigg)>0,
\ee
which holds for small enough $\rho$. This completes the proof that $\cP$ is a contraction operator.

We now proceed to establish the $\ell^\infty$ stability of the scheme. Let $\{U_j^{n-1}\}_{j=1}^J$ be the solutions to \eqref{eq:implicit}. By expressing the discrete operators  $A^{\a}_{\eps,h,k}$ and $K^{\a,1}_{\eps,h}$ in the monotone form \eqref{eq:A_mono} and \eqref{eq:K_mono}, and  substituting them into \eqref{eq:implicit}, we obtain that
\begin{align*}
[1+2d\theta +&\Delta t\big(\sum_{m\in \Z^d}d^{\a,n}_{h,k,m,i}+ \sum_{m\not =0}\kappa^{\a,n}_{h,m,i}\big)]U^n_{j,i}-\Delta t\big(\sum_{m\in \Z^d}d^{\a,n}_{h,k,m,i}U^n_{j,m}+\sum_{m\not =0}\kappa^{\a,n}_{h,m,i}U^{n}_{j,i+m}\big)\\
&=U^{n-\f{1}{2}}_{j,i}+\Delta t
 \tilde{f}^{\a}(\x^n_{j,i},U^{n}_{j,i},\Delta U^{n}_{j,i}, B^{\a}_{\eps,h}U^{n}_{j,i})+2d\theta U^{n}_{j,i},
\end{align*}
from which we can deduce
\begin{align}
[1+&2d\theta+\Delta t\big(\sum_{m\in \Z^d}d^{\a,n}_{h,k,m,i}+ \sum_{m\not =0}\kappa^{\a,n}_{h,m,i}\big)]U^n_{j,i}-\Delta t\big(\sum_{m\in \Z^d}d^{\a,n}_{h,k,m,i}+\sum_{m\not =0}\kappa^{\a,n}_{h,m,i}\big)|U^{n}_{j,i}|_0\nb\\
&\le \Delta t\big[
 f^{\a}(\x^n_{j,i},U^{n}_{j,i},\Delta U^{n}_{j,i},B^{\a}_{\eps,h}U^{n}_{j,i})- f^{\a}(\x^n_{j,i},0,\Delta U^{n}_{j,i},B^{\a}_{\eps,h}U^{n}_{j,i})\big]
 \l{eq:stab_line1}\\
 &+\Delta t\big[f^{\a}(\x^n_{j,i},0,\Delta U^{n}_{j,i},B^{\a}_{\eps,h}U^{n}_{j,i})- f^{\a}(\x^n_{j,i},0,\Delta U^{n}_{j,i},0)
 \l{eq:stab_line2}\big]\\
&  +|U^{n-\f{1}{2}}_j|_0+( \Delta t [\tilde{f}^{\a}(\x^n_{j,i},0,\Delta U^{n}_{j,i},0)-\tilde{f}^{\a}(\x^n_{j,i},0,0,0)]+2d\theta U^{n}_{j,i})+\Delta t \tilde{f}^{\a}(\x^n_{j,i},0,0,0)
.\nb
\end{align}

Using  similar arguments as those for the upper bound of  \eqref{eq:ctrc_line1},
we deduce that \eqref{eq:stab_line1} is bounded above by $-\Delta tCU^n_{j,i}$ independent of the sign of $U^n_{j,i}$.

Suppose now $B^{\a}_{\eps,h}U^{n}_{j,i}<0$, then we obtain from the monotonicity of $f$ in $k$  that $\eqref{eq:stab_line2}$ is nonpositive. Then the $\ell^\infty$ stability of the numerical flux and the
boundedness of $f^{\a}(\x,0,0,0)$ yield that
\begin{align}\l{eq:Un&n-1/2}
(1+\Delta tC)|U^n_{j}|_0\le |U^{n-\f{1}{2}}_{j}|_0+ \Delta tC_1.
\end{align}
Here $C$ is the constant from Assumption \ref{assum:coeff} and $C_1>0$ is 
a large enough constant that we will choose later.
On the other hand, if $B^{\a}_{\eps,h}U^{n}_{j,i}>0$, the Lipschitz continuity of $f$ in $k$ enables us to bound  $\eqref{eq:stab_line2}$ by $C B^{\a}_{\eps,h}U^{n}_{j,i}$, which along with \eqref{eq:K_mono} implies again \eqref{eq:Un&n-1/2}.


With the estimate \eqref{eq:Un&n-1/2} in hand, we are ready to derive a uniform bound for the solutions 
$\{U_j^n\}$, which is independent of $h$ and $c$. The proof follows from an inductive argument.
Let us introduce  the notation $|U^n|_0=\max_{1\le j\le J}|U_j^n|_0$ for each $n$ and define the term $a_0=\max(|g|_0,|\zeta|_0)$, then it is clear that $a_0\ge \max(|U^0|_0,|\zeta|_0)$. Suppose we have $a_{n-1}$ such that $a_{n-1}\ge \max(|U^{n-1}|_0,|\zeta|_0)$.  Then the definition of $U^{n-\f{1}{2}}_{j,i}$ implies that $|U^{n-\f{1}{2}}_j|_0\le \max(|\zeta|_0,|U^{n-1}|_0)\le a_{n-1} $. Define the term
$$
a_{n}\coloneqq \f{1}{1+\Delta t C}a_{n-1}+\Delta t C_1,
$$
with the same constants as those in \eqref{eq:Un&n-1/2}, then we have $|U^{n}|_0\le a_n$. To proceed by induction, we further require $a_n\ge |\zeta|_0$. Since $a_{n-1}\ge |\zeta|_0$ and $C$ is fixed, it suffices to require $C_1\ge C|\zeta|_0$. In this way, we can construct a sequence $\{a_n\}$, such that $|U^{n}|_0\le a_n$, but $a_n$ is uniformly bounded independent of $c$, $h$ and $\Delta t$, and hence this completes the proof of $\ell^\infty$ stability.

We now study the weak monotonicity of the scheme. Let $V^n_{j,i}\ge U^n_{j,i}$ and $\tilde{V}^{n}_{k,i(j)}\ge \tilde{U}^{n}_{k,i(j)}$ for all  $i,j,k,n$, then we have $V^{n+\f{1}{2}}_{j,i}\ge U^{n+\f{1}{2}}_{j,i}$. Moreover the monotonicity of $f$ in $k$ and the weak monotonicity of $\tilde{f}$ imply that
$$
\sum_{m\in \Z^d}d^{\a,n}_{h,k,m,i}U^{n+1}_{j,m}+ \sum_{m\not =0}\kappa^{\a,n}_{h,m,i}U^{n+1}_{j,i+m}+ \tilde{f}^{\a}(\x^{n+1}_{j,i},U^{n+1}_{j,i}, \Delta U^{n+1}_{j,i},\sum_{m\not =0}\b^{\a,n}_{h,m,i}[U^{n+1}_{j,i+m}-U^{n+1}_{j,i}])
$$
is nondecreasing with $\{U^{b+1}_{j,a}\}_{\textnormal{$(a,b)\!\not =\!(i,n)$}}$, which gives the weak monotonicity of the scheme  \eqref{eq:implicit}.

Finally we study the consistency of the scheme. By using the Lipschitz continuity of $x\to \min(x,a)$, it is clear that it suffices to bound 
\begin{align*}
(I_1)\coloneqq & \Delta t\big(A^{\a}_{\eps,h,k}\phi^{n+1}_{j,i}+K^{\a,1}_{\eps,h}\phi^{n+1}_{j,i}+
 \tilde{f}^{\a}(\x^{n+1}_{j,i},\phi^{n+1}_{j,i}+\xi,\Delta \phi^{n+1}_{j,i}, B^{\a}_{\eps,h}\phi^{n+1}_{j,i})\big)\\
(I_2)\coloneqq&\bigg|\f{\phi^{n+1}_{j,i}-\phi_{j,i}^{n}}{\Delta t}-\big(A^{\a}_{\eps,h,k}\phi^{n+1}_{j,i}+K^{\a,1}_{\eps,h}\phi^{n+1}_{j,i}+
 \tilde{f}^{\a}(\x^{n+1}_{j,i},\phi^{n+1}_{j,i}+\xi,\Delta \phi^{n+1}_{j,i},B^{\a}_{\eps,h}\phi^{n+1}_{j,i})\big)\\
 &-\phi_{j,t}(\x_{j,i}^{n+1})-A_\eps^{\a} \phi_j(\x_{j,i}^{n+1})-K_\eps^{\a} \phi_j(\x_{j,i}^{n+1})-
 f^{\a}(\x^{n+1}_{j,i},\phi(\x^{n+1}_{j,i}),D\phi(\x^{n+1}_{j,i}),B_\eps^{\a}\phi(\x^{n+1}_{j,i}))\bigg|,
\end{align*}
which can be estimated by using Lemma \ref{lemma:consistency_op}, Lemma \ref{lemma:lax}, and the Lipschitz continuity of $f$.
\end{proof}

\begin{Remark}
The contraction operator $\cP$ is introduced to demonstrate our scheme admits a unique solution for any given discretization parameters $\Delta t$, $h$, $k$ and $\eps$. However, due to its low convergence rate, it is not advisable to implement this contraction mapping directly to solve the nonlinear equation \eqref{eq:implicit}. In fact, 
Lemma \ref{lemma:sum_coeff} and the stability condition  \eqref{eq:P} restrict the contraction constant of $\cP$ to admit a lower bound depending on the spatial discretization of the diffusion operator. This undesirable dependence of $\Delta t$ on $k$ can be avoided  by considering the mapping $T$ defined by \eqref{eq:contract_implement}, which is implicit in the local terms. It has been shown that for small enough $h$, the contraction constant of $T$ is proportional to $\theta$, which can be chosen to achieve a rapid convergence. 

\end{Remark}

%

\section{Numerical experiments}\l{sec:numerical}
In this section, we  present several numerical experiments to analyse the effectiveness of the numerical scheme proposed in Section \ref{sec:scheme}. We shall investigate the convergence of numerical solutions with respect to the switching cost, timestep, and mesh size, and show that a relatively coarse discretization of the admissible control set already leads to an accurate approximation.

We consider a portfolio optimization problem over a time interval $[0,T]$, in a framework of recursive utility. An investor can control his wealth process $X^{t,x,\a}$ through a selection of the control process $\a\in \cA^t_t$, say his or her portfolio strategy, and can also choose the duration of the investment via a stopping time $\tau$. If the agent chooses a strategy pair $(\a,\tau)$, then the associated terminal reward is given by
$$
\xi^{t,x,\a}_\tau=\zeta(\tau,X_\tau^{t,x,\a})1_{t\le \tau<T}+g(X_\tau^{t,x,\a})1_{\tau=T}
$$
for some utilities $\zeta$ and $g$, and where {$\tau\in\cT_t^t$, the set of $\mathbb{F}^t$-stopping times valued in $[t,T]$.}

The performance of this investment is evaluated under a 
particular nonlinear expectation, called the  recursive utility process (see e.g.\ 
\cite{chen2002}), which is associated with a BSDE (with Lipschitz continuous drivers). 
It generalizes  the standard additive utilities by including a dependence on the future utility (corresponding to the future wealth).
Roughly speaking, the recursive utility depends on the future utility through the dependance of the driver $f$ on $y$, and can also depend on the ``variability" or ``volatility" of future utility through the dependance of $f$ on $z$ and $k$.

Let $x$ be the wealth at the initial time $t$, $(\a,\tau)$ be the chosen strategy, $\cE^{\a,t,x}[\cdot]$ be a recursive utility function associated with the BSDE with driver $f^\a$.
 The aim of the investor is to maximize the utility of the investment: 
 $$
u(t,x)\coloneqq \sup_{\tau\in \cT_t^t}\sup_{\a\in \cA_t^t} \cE^{t,\a}_{t,\tau}[\xi^{t,x,\a}_\tau],
$$
 over all admissible choices of $(\a,\tau)$. Under Assumption \ref{assum:coeff}, it can be shown  that the value function $u$ of this mixed optimization problem
coincides with the unique bounded viscosity solution of the (backward) HJBVI \eqref{eq:hjbvi}.

For the numerical tests, we  consider a financial market with a risk-free asset  with an interest rate $r$ and a risky asset whose price follows
$$
dS_t=S_{t^-}\bigg[b \, dt+\sigma \, dW_t+\int_E \eta(e)\,\tilde{N}(dt,de)\bigg],
$$
where $W$ is a Brownian motion and $\tilde{N}(dt,de)=N(dt,de)-\nu(de)dt$ is a compensated jump measure. If we denote by $\a_t$  the percentage of the portfolio held in the risky asset at time $t$, then the dynamics of the portfolio is given by
$$
dX_t=\a_tX_{t^-}\bigg[b \, dt+\sigma \, dW_t+\int_E \eta(e)\,\tilde{N}(dt,de)\bigg],\q X_0=x_0.
$$
The performance will be evaluated by the
  recursive utility function induced by the BSDE with the following driver:
 $$
f(t,x,y,z)=\psi(t,x)-\beta y-\kappa|z|.
$$
for some instantaneous reward function $\psi$. 
{
Recall that any  concave utility function admits a dual representation via a set of probability measures    absolutely continuous with respect to the original probability measure $P$ (see e.g. \cite{karoui1994}).  This result allows us to interpret $\kappa\ge 0$ 
as an  ambiguity-aversion coefficient
relative to the Brownian motion as suggested in \cite[Section~3.3]{chen2002}.}

The value function of this control problem satisfies the following HJBVI: 
\bb\l{eq:hjbvi_ex1}
\begin{cases}
\min\big\{
u(t,x)-\zeta(t,x),u_t+\inf_{\a\in [0,1]}\big(-L^\a u-\psi+\b u+ \a\kappa\sigma |x u_x|)\big)
\big\}=0,\\
u(0,x)-g(x)=0
\end{cases}
\ee
for $(t,x)\in [0,T]\t \R$,
where the nonlocal operator $L^\a= A^\a+K^\a$  satisfies for $\phi\in C^2([0,T]\t\R)$
\begin{align}
A^\a\phi(t,x)&= \f{1}{2}\a^2\sigma^2 x^2\phi_{xx}(t,x)+(\a b + (1-\a) r) x \phi_x(t,x),\nb
\\
K^\a\phi(t,x)&=\int_{\R\setminus\{0\}}\big(\phi(t,x+\a x\eta(e))-\phi(t,x)-\a x\eta(e) \phi_x(t,x)\big)\,\nu(de). \l{eq:k_ex1}
\end{align}

We then specify the choice of data for our numerical experiments. We use the exponential utility function $\zeta(t,x)=g(x)=(1-e^{-x})^+$, which determines both the
intermediate and terminal payoff, and acts as the initial condition and the obstacle to the HJBVI.
Moreover, we consider the tempered stable L\'{e}vy measure $\nu(de)=\f{e^{-\mu |e|}}{|e|}de$ on $\R$ with intensity $\eta(e)=1\wedge |e|$ for the jump component 
(which is a special case of the variance Gamma model in \cite{cont2005}). For simplicity, we choose a zero interest rate, i.e., $r=0$.

We further choose the function $\psi(t,x)=0.8\exp(-(T-t))\exp(-x/2)$ as the instantaneous reward. 
As we will see later, this choice of $\psi$ implies that the optimal control $\a$ varies in the state space and evolves in time,
and there can be non-trivial stopping.
The resulting HJBVI will be localized to the domain $(0,2)$ with 
$u(t,x)=g(x)$ for $(t,x)\in (0,T)\t \R\setminus (0,2)$.
The numerical values for the parameters used in the experiments are given in Table \ref{table:ex1parameter}.

\bigskip
\begin{tablehere}
\centering
\begin{tabular}{||c|c|c|c|c|c|c||}\hline
$\b$ &  $\kappa$ & $b$ & $\sigma$  & $\mu$ &$T$&$x_0$\\ \hline
0.2 & 1 & 0.1& 0.15&  6& 1&1\\\hline
\end{tabular}
\caption{Model parameters for the recursive utility maximization problem.}
\label{table:ex1parameter}
\end{tablehere}
\bigskip

Now we are ready to  discuss the  selection of the discretization parameters in detail. The density of the tempered stable measure $\nu$ enables us to improve the estimates in Lemma \ref{lemma:sum_coeff} to 
$\sum_{m\not =0}\kappa^{\a,n}_{h,m,i}\le \log(\eps)$, and hence choosing $\eps=h$ and $\Delta t=O(h)$  leads us to a consistent approximation to the switching system \eqref{eq:hjbvi_s}. 
Moreover, choosing $\theta=\f{1}{40}$ and $\Delta t=\f{h}{15}$
 ensures the numerical flux is stable and the contraction constant of  $T$ in \eqref{eq:contract_implement} is less than $\f{1}{10}$. 
 {The
coefficients of the nonlocal terms are evaluated by the midpoint quadrature formula, which is clearly monotone and consistent.}
 We observe that for the control problem with the parameters as in Table \ref{table:ex1parameter}, the optimal strategy $\a^*$ will always be obtained at one of the endpoints of $[0,1]$.
In fact, using Taylor's theorem, we are able to approximate the nonlocal term $K^\a u$ by
$$
K^\a u(t,x)\approx \f{1}{2}\a^2x^2 \int_{\R\setminus\{0\}} (1\wedge |e|)^2 \,\nu(de) u_{xx}(t,x)=\f{1}{2}\a^2x^2 C u_{xx}(t,x),
$$
at any given $(t,x)$ for which the value function lies above the obstacle and is sufficiently smooth.
Then we infer from the HJBVI \eqref{eq:hjbvi_ex1} that the optimal control $\a^*$ is the  maximizer of a quadratic function on $[0,1]$, which is attained in the interior only if 
$$
u_{xx}(t,x)<0, \q -\f{bu_x-\sigma |u_x|}{(\sigma^2+C)xu_{xx}}\in (0,1).
$$
However, since we have $b<\sigma$, the above conditions can never hold for any $x>0$. Consequently, we deduce that the admissible set is already finite, and replacing $[0,1]$ by $\bA_\delta=\{0,1\}$ in \eqref{eq:hjbvi_ex1} will not introduce any discretiztion error. This has been  confirmed with our numerical experiments. For the sake of simplicity, we  discretise each component of the switching system on a single uniform mesh, thus Condition \ref{assum:scheme} (1) is trivially satisfied. 

Table \ref{table:ex2.1_mesh} contains the numerical solutions to the  last component of the switching system at the grid point $(T,x_0)$ with different mesh size $h$ and switching cost $c$. We  examine the convergence of the numerical solutions, denoted as $U_h$,  in $h$ for  fixed $c$, as well as 
their convergence with respect to the cost $c$.  For any fixed positive switching cost $c$, we infer from the lines (a) that the numerical solutions converge monotonically to the exact solution. Moreover, the lines (c) indicate the approximation error admits an asymptotic  magnitude $O(h)+O(\Delta t)$, which  seems not to be affected by the size of the cost $c$. 
By considering the boldface values in Table \ref{table:ex2.1_mesh}  as an accurate approximation to the exact solution of the switching system with a given cost $c$, 
we can further conclude
that the switching system is consistent to the HJBVI \eqref{eq:hjbvi_ex1} with order $1$. This follows from the approximate factor of four between the differences $0.00469$, $0.00117$, and $0.00029$ between the last three pairs of values, proportional to the reduction in $c$.
Therefore, by taking $c=O(h)$ and $\Delta t=O(h)$, we can obtain a first-order scheme for the HJBVI.


\begin{table}
\centering
\begin{tabular}{||l|l|c|c|c|c|c|c|c||}\hline
$h$ &   & 1/25& 1/50 & 1/100 & 1/200 & 1/400 & 1/800 &1/1600 \\ \hline
$c=\f{1}{40}$ & (a)   &  
0.77538   &  0.77770  & 0.77869   &0.77912   &  
0.77933 &0.77942  & \textbf{0.77947} \\
 & (b) &  & 2321.35  & 984.37  &  437.19 & 203.03 &  97.30  & 47.53\\
 & (c) &  &   &  2.3582 &   2.2516 &   2.1533 &  2.0867 &  2.0471  \\\hline
$c=\f{1}{160}$ &  (a)&  
0.79003   & 0.79194   & 0.79290   & 0.79338   & 0.79361
&  0.79373 & \textbf{0.79379} \\
 & (b) &  &  1916.14 & 954.91  & 476.01  & 237.63 & 118.73   & 59.34
 \\
& (c) &  &   &  2.0066 &     2.0061 &    2.0031 &    2.0015 & 2.0008
\\\hline
$c=\f{1}{640}$ &  (a)  &  
0.79471  &  0.79663 &   0.79759 & 0.79806   & 0.79830
 & 0.79842 &  \textbf{0.79848}   \\
 & (b)   &  & 1917.40 & 955.54 &   476.33 &   237.79 & 118.80 &    59.38\\
  & (c)   &  &   &  
  2.0066  &   2.0061  &  2.0031   &  2.0015 &     2.0008\\
 \hline
 $c=\f{1}{2560}$ & (a)   &  
 0.79588   & 0.79780   & 0.79876  & 0.79923  &0.79947  & 0.79959 & 
 \textbf{0.79965} \\
  &(b)   &    &   1917.71&  955.70 &   476.40 &   237.83 &   118.82 &    59.39\\
  &(c)   &    &  &  2.0066 &    2.0061 &     2.0031&     2.0015 &     2.0008 \\\hline
  $c=\f{1}{10240}$ & (a)   &  
 0.79618   &  0.79810  &0.79905   & 0.79953  &0.79977& 0.79988 &  \textbf{0.79994} \\
  &(b)   &    &   1917.79&  
  955.74 &   476.42 &   237.84 & 118.83 &   59.39\\
  &(c)   &    &  &  2.0066 &     2.0061&     2.0031 &    2.0015
  & 2.0007
   \\\hline
\end{tabular}
\caption{Numerical solutions for  the recursive utility maximization problem with different mesh sizes and switching costs. Shown are: (a) the numerical solutions $U_h$ at $(T,x_0)$; (b) the increments $U_{h}-U_{2h}$ (in $ 10^{-6}$) ; (c) the rate of increments $(U_{h}-U_{2h})/(U_{2h}-U_{4h})$.}
\label{table:ex2.1_mesh}
\end{table}

We then proceed to analyze the effect of the control discretization. We pick the same parameters as those in Table \ref{table:ex1parameter}, except that $b=0.25$, which is chosen such that 
it is now possible that the optimal control is attained in the interior of $(0,1)$ (as seen from a similar argument as earlier). 
Computations are performed using \textsc{Matlab} R2016b on a 3.30GHz Intel Xeon ES-2667 16-Core processor with 256GB RAM to enable parallelization.  Table \ref{table:ex2.1_ctrl} illustrates the numerical results for different control meshes ($J=1/\delta+1$) with a fixed mesh size $h=0.005$ and switching cost $c=1/2560$, and also compares   the runtime with or without parallelization.

We can clearly observe from line (a) second order convergence of the numerical solutions, and a relatively coarse control mesh has already yielded an accurate approximation with a negligible control discretization error. 

Next, we discuss lines (b)--(f) which analyse the algorithm's parallel efficiency.
Hereby, the implicit finite difference scheme for individual components of the switching system (i.e., \eqref{eq:PCCT}, for different $j$) is solved independently on different processors,
while the maximisation step \eqref{eq:Un+1/2} requires communication between processors.

The total execution time with and without parallelization are presented in line (b) and (c), respectively, which indicate a significant reduction of computational times. 
Moreover, by subtracting the communication time among clusters, as shown in line (d), from the total runtime, we can obtain the actual time spent on executing the numerical scheme (line (e)). The speed-up rate of the parallelization is shown in line (f), which grows with the number of controls, and remains stable at the number of cores. Therefore, together with parallelization,  piecewise constant timestepping enables us to achieve a high accuracy in the control discretization without significantly increasing the computational time, which is an advantage over policy iterations, which do not parallelise naturally.


We finally examine the impact of the computational domain by performing computations on  $(0,3)$ with $h=1/400$,  $\Delta t=h/20$, $c=1/640$ and the parameters as in Table \ref{table:ex1parameter}. Compared to the results in Table \ref{table:ex2.1_mesh}, this larger domain leads to a relative difference of $7.53 \cdot 10^{-7}$, which is negligible compared to the time and spatial discretization errors.

\bigskip
\begin{tablehere}
\centering
\begin{tabular}{||l|c|c|c|c|c|c||} \hline
$J$ &&  2 & 11 & 21  & 41  & 81\\ \hline 
average error ($\times 10^{-7}$) &(a)& 1.1392 & 0.2398 &  0.0570  &  0.0133  &  0   \\\hline
runtime in seconds &(b) & 200.014 & 1106.266   &   2142.714  & 4204.336   & 8536.232   \\
 &(c) & 155.977 & 228.594    & 331.294     & 501.895  & 809.178  \\
 &(d) &53.728 & 124.390   &  192.634  & 230.019  & 262.488   \\
&(e) & 102.249 & 104.204   &  138.660  & 271.876   & 546.690  \\\hline
 speed-up rate (b/e)& (f)& 1.9561 & 10.61   &  15.45  & 15.46   & 15.61  \\\hline
\end{tabular}
\caption{Numerical results for  the recursive utility maximization problem with different control refinements. Shown are: (a) the difference to the solution with the finest refinement; (b) total runtime time without parallelization; (c) total runtime time with parallelization on 16 processors; (d) total communication time among processors;
(e) net parallel computation time without communication; (f) speed-up rate of parallelization. 
}
\label{table:ex2.1_ctrl}
\end{tablehere}
\bigskip

The numerical value function and the corresponded feedback control strategy with $J=21$ are presented in Figure \ref{fig:valuefunction}, in which the white area represents the region where the obstacle is active, and otherwise the colour indicates the value of the optimal control, as shown in the panel on the right.
The approximation to the optimal control pair $(\tau,\alpha)$ was found from the numerical solution as follows (see also \eqref{eq:Un+1/2} and Remark \ref{rem:controls}),
noting that in our tests $x_{j,i}=x_{k,i}$ for all $j,k$, and therefore no interpolation is needed:
\begin{eqnarray*}
i_n^* &\in& {\rm argmax}_{k}{U}^{n}_{k,i}, \\
\theta_i^n &=& \left\{
\begin{array}{rl}
0 & \max_{k}{U}^{n}_{k,i} > \xi(t_n,x_{1,i}), \\
1 & \max_{k}{U}^{n}_{k,i} \le \xi(t_n,x_{1,i}),
\end{array}
\right.
\end{eqnarray*}
where $\alpha_i^n = \alpha_{i_n^*}$ is an approximation to the optimal policy and $\{(t_n,x_{1,i}): \theta_i^n=1\}$ is an approximation to the stopping region.

\bigskip
\begin{figurehere}
    \centering
    \includegraphics[width=0.49\columnwidth,height=6cm]{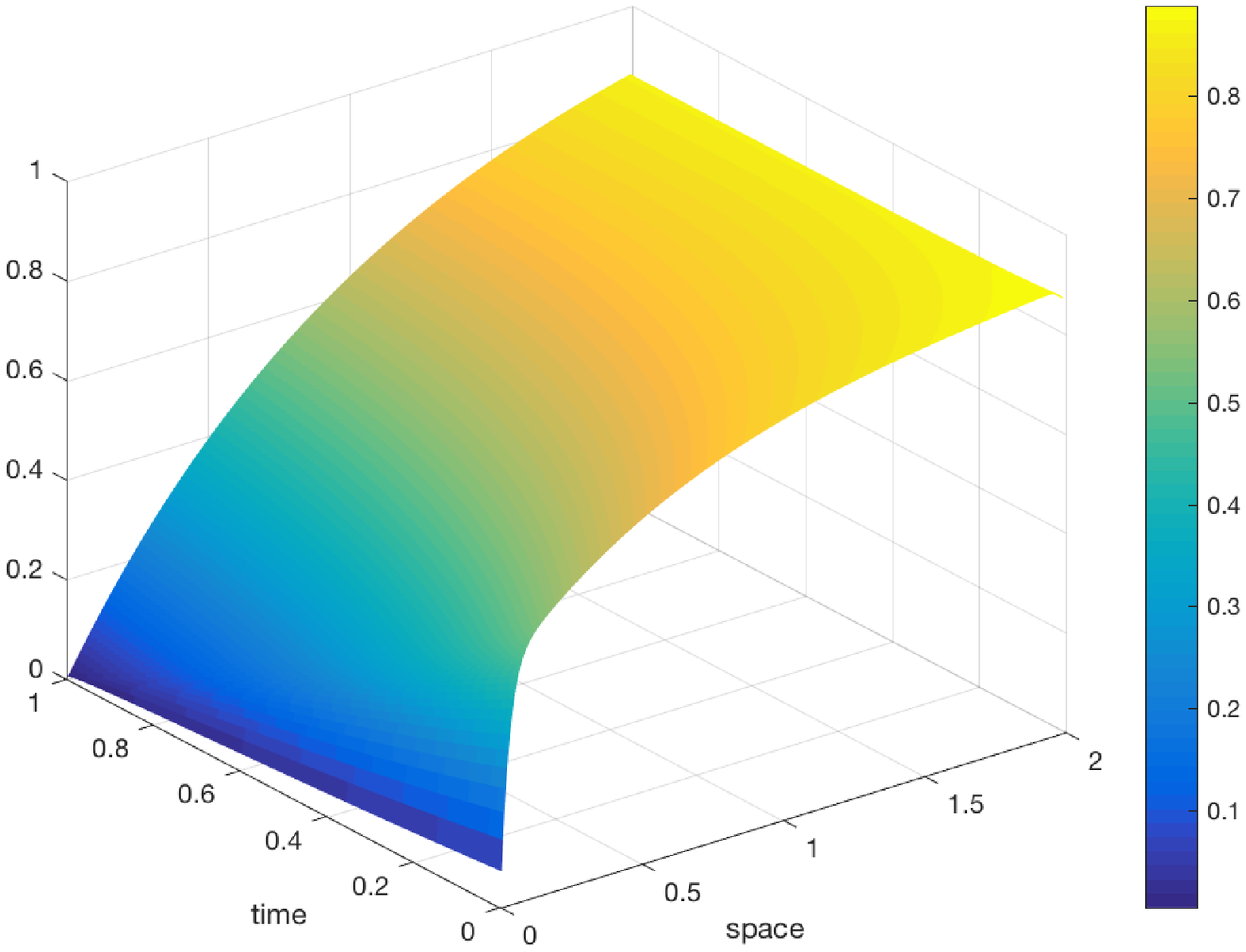}\hfill
    \includegraphics[width=0.49\columnwidth,height=6cm]{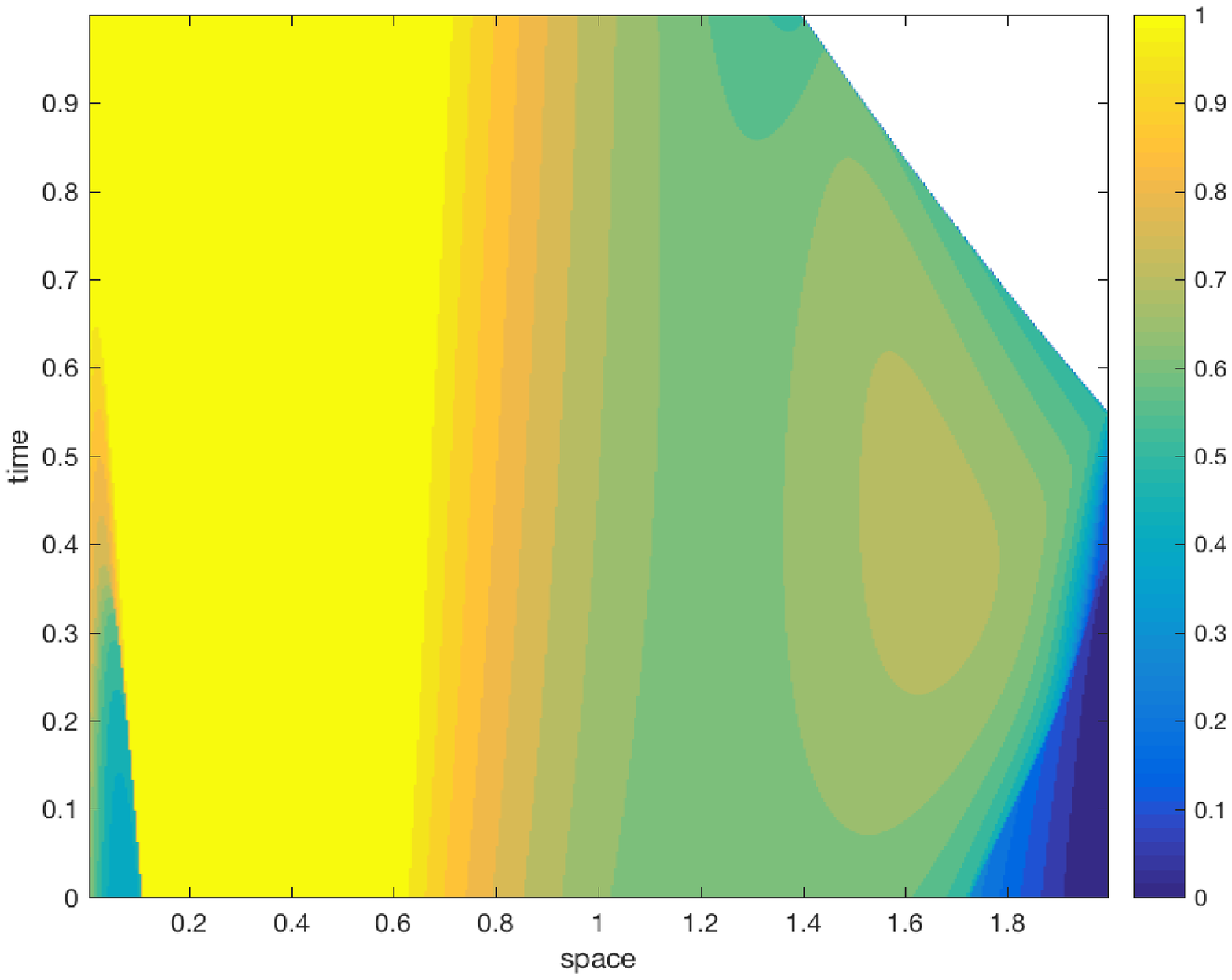}
    \caption{Numerical value functions (left) and corresponding control strategies with $J=21$ (right), where the early stopping region is white.}
    \label{fig:valuefunction}
 \end{figurehere}
 \bigskip


\section{Conclusions}

This paper provides a PDE approximation scheme for the value function of a mixed stochastic control/optimal stopping problem with nonlinear expectations and infinite activity jumps, which is the unique {viscosity} solution of a nonlocal HJB variational inequality. The approach that we have adopted is based on piecewise constant policy time stepping (PCPT), which reduces the problem to a system of semi-linear PDEs, and a monotone approximation scheme.
{We prove the convergence of the numerical scheme} and illustrate  the theoretical results with some numerical examples in the case of a recursive utility maximisation problem. 

To the best of our knowledge, this is the first paper which proposes a numerical approximation for a control problem in such a generality. 
Natural next steps would be to derive theoretical results on the convergence rate and to extend this approach to the case of Hamilton-Jacobi-Bellman-Isaac equations obtained in \cite{buckdahn2008}.

\appendix


\section{Consistency of control discretization}
\label{app:cont-disc}

\begin{proof}
For any given $\x\in \cQ_T$, $\delta>0, \xi\in \R$ and  $\phi\in C^{1,2}(\bar{\cQ}_T)$, 
by using the expressions of $F^\eps$ and $F^{\eps,\delta}$ and the fact that $y\to \min(a,y)$ is Lipschitz continuous with  constant 1, we have
\begin{align}
|F&^\eps(\x,\phi(\x),D\phi(\x),D^2\phi(\x))-F^{\eps,\delta}(\x,\phi(\x)+\xi,D\phi(\x),D^2\phi(\x))|\nb\\
&\le |\xi|+|\sup_{\a\in \bA}\cH^\a \phi(\x)-\sup_{\a\in \bA_\delta}\cH^\a (\phi+\xi)(\x)|\nb\\
&\le|\xi|+ |\sup_{\a\in \bA}\cH^\a \phi(\x)-\sup_{\a\in \bA}\cH^\a (\phi+\xi)(\x)|+|\sup_{\a\in \bA}\cH^\a (\phi+\xi)(\x)-\sup_{\a\in \bA_\delta}\cH^\a (\phi+\xi)(\x)|,\l{eq:F-Fdel}
\end{align}
where we denote for simplicity $\cH^\a \phi(\x_0)=L_\eps^\a \phi(\x_0)+f^\a(\x_0,\phi(\x_0),\tilde{\sigma}^\a\cdot D\phi(\x_0),B_\eps^\a \phi(\x_0))$. 

We first estimate the second term in \eqref{eq:F-Fdel}. Assumption \ref{assum:coeff} implies that $\cH^\a \phi(\x)$ and $\cH^\a (\phi+\xi)(\x) $ are continuous in $\a$. Hence the supremum of  $\cH^\a \phi(\x)$ and  $\cH^\a (\phi+\xi)(\x)$ on the compact set $\bA$ are attained at $\a^*$ and $\a^*_\xi$, respectively.  Moreover, we deduce from \eqref{eq:a}, \eqref{eq:k} and \eqref{eq:b} that $A_\eps^\a\phi(\x)=A_\eps^\a(\phi+\xi)(\x)$, $B_\eps^\a\phi(\x)=B_\eps^\a(\phi+\xi)(\x)$ and $K_\eps^\a\phi(\x)=K_\eps^\a(\phi+\xi)(\x)$, and consequently
\begin{align*}
\sup_{\a\in \bA}&\cH^\a (\phi+\xi)(\x)-\sup_{\a\in \bA}\cH^\a \phi(\x)\le \cH^{\a_\xi^*} (\phi+\xi)(\x)-\cH^{\a_\xi^*} \phi(\x)\\
&=f(\a_\xi^*,\x,\phi(\x)+\xi,\tilde{\sigma}^{\a_\xi^*}\cdot D\phi(\x),B_\eps^{\a_\xi^*} \phi(\x))-f(\a_\xi^*,\x,\phi(\x),\tilde{\sigma}^{\a_\xi^*}\cdot D\phi(\x),B_\eps^{\a_\xi^*} \phi(\x))\le C|\xi|,
\end{align*}
where for the last inequality we used the Lipschitz continuity of $f$ in $u$. By reversing the roles of $\sup_{\a\in \bA}\cH^\a \phi(\x)$ and $\sup_{\a\in \bA}\cH^\a (\phi+\xi)(\x)$, we  get 
\bb\l{eq:phi-phixi}
|\sup_{\a\in \bA}\cH^\a \phi(\x)-\sup_{\a\in \bA}\cH^\a (\phi+\xi)(\x)|\le C|\xi|.
\ee

We then consider the last term in \eqref{eq:F-Fdel}. The construction of $\bA_\del$ implies that there is $\a^*_\delta\in\bA_\del$ such that $|\a_\xi^*-\a^*_\delta|<\delta$. Moreover, by the compactness of $\bA$ and the Lipschitz continuity of coefficients in $\x$ uniformly in $\a$, we deduce there exists a function $\omega_0$ such that at  $\x=(t,x)$,
\begin{align*}
&|\tilde{\sigma}^{\a_\xi^*}(\tilde{\sigma}^{\a_\xi^*})^T-\tilde{\sigma}^{\a^*_\delta}(\tilde{\sigma}^{\a^*_\delta})^T|+|b^{\a_\xi^*}-b^{\a_\delta^*}|+|\tilde{\sigma}^{\a_\xi^*}-\tilde{\sigma}^{\a_\delta^*}|\le \omega_0(\x,\delta),\\
&|(\eta^{\a_\xi^*}(x,e))^T(\eta^{\a_\xi^*}(x,e))-(\eta^{\a^*_\delta}(x,e))^T\eta^{\a^*_\delta}(x,e)|\le \omega_0(\x,\del)(1\wedge |e|^2),\\
&|\eta^{\a_\xi^*}(x,e)-\eta^{\a^*_\delta}(x,e)|\le\omega_0(\x,\del)(1\wedge |e|),
\end{align*}
and the function $\omega_0(\x,\del)$ is locally Lipchitz continuous with $\x$ for all $\delta\le 1$, 
and satisfies $\omega_0(\x,\del)\to 0$ as $\del\to 0$. These estimates lead us to 
the inequality
$$
|A_\eps^{\a_\xi^*}(\phi+\xi)(\x)-A_\eps^{\a_\del^*}(\phi+\xi)(\x)|\le \omega_0(\x,\del)\max(|D\phi(\cdot)|_{B(\x,1)}|,|D^2\phi(\cdot)|_{B(\x,1)}|),
$$
and also together with the Taylor expansion and the fact $|\eta|\le C$,  enable us to derive that 
\begin{align*}
|&K_\eps^{\a_\xi^*}(\phi+\xi)(\x)-K_\eps^{\a_\del^*}(\phi+\xi)(\x)|\\
&\le 
\int_E \bigg|(\eta^{\a_\xi^*}(x,e))^T\bigg(\int_0^1 (1-s)D^2\phi(t,x+\eta^{\a_\xi^*}(x,e))\,ds\bigg)\eta^{\a_\xi^*}(x,e) \\
&-(\eta^{\a_\del^*}(x,e))^T\bigg(\int_0^1 (1-s)D^2\phi(t,x+\eta^{\a_\del^*}(x,e))\,ds\bigg)\eta^{\a_\del^*}(x,e)\bigg|\,\nu(de)\le C\omega_0(\x,\del)|D^2\phi(\cdot)|_{B(\x,C)},\\
|&B_\eps^{\a_\xi^*}(\phi+\xi)(\x)-B_\eps^{\a_\del^*}(\phi+\xi)(\x)|
\le
\int_E \bigg|\bigg((\eta^{\a_\xi^*}(x,e))^T\int_0^1 \phi(t,x+\eta^{\a_\xi^*}(x,e))\,ds\bigg)\\ 
&-\bigg((\eta^{\a_\del^*}(x,e))^T\int_0^1 D^2\phi(t,x+\eta^{\a_\del^*}(x,e))\,ds\bigg)\bigg|\,\nu(de)\le C\omega_0(\x,\del)|D\phi(\cdot)|_{B(\x,C)},
\end{align*}
and consequently we obtain from the Lipschitz continuity of $f$ that 
\begin{align*}
\sup_{\a\in \bA}\cH^\a (\phi+\xi)(\x)&-\sup_{\a\in \bA_\del}\cH^\a (\phi+\xi)(\x)\le \cH^{\a_\xi^*} (\phi+\xi)(\x)-\cH^{\a_\del^*} (\phi+\xi)(\x)
\le \omega_1(\x,\del),
\end{align*}
for a suitable defined $\omega_1(\x,\del)$ with the properties of $\omega_0(\x,\del)$.
Therefore, using \eqref{eq:F-Fdel}, \eqref{eq:phi-phixi} and the fact that $\bA_\del\subset\bA$, we have
$$
 |\sup_{\a\in \bA}\cH^\a \phi(\x)-\sup_{\a\in \bA_\delta}\cH^\a (\phi+\xi)(\x)|\le \omega_1(\x,\del)+C|\xi|,
$$
which  completes the proof of our desired result.
\end{proof}

\section{Comparison principle for switching systems}\l{sec:comparison}
In this section, we  establish the comparison principle for switching system \eqref{eq:hjbvi_s}, cf.\ Theorem~\ref{thm:comparision_s}. 
We consider a slightly more general switching system with no truncation of the singular measure in $K_\eps$ and $B_\eps$, which includes as a special case the switching system  \eqref{eq:hjbvi_s}.
We first use a classical no-loop argument 
to reduce the problem into scalar cases, and then analyze the scalar HJBVI by extending the results for continuous  solutions in \cite{dumitrescu2015} to semicontinuous  viscosity solutions. 
For  simplicity,  we denote by $\sigma$ the modified diffusion coefficient $\tilde{\sigma}^\a$ defined as \eqref{eq:small_diff}.

\begin{proof}[Proof of Theorem \ref{thm:comparision_s}]
Set $$M=\sup_{1\le j\le J,\x,\y\in \cQ_T}(U_j(\x)-V_j(\y)).$$
It suffices to show that $M\le 0$. For any given $\eps,\rho>0$, we introduce the functions
\bb
\psi_j^{\eps,\rho}(t,s,x,y)=U_j(t,x)-V_j(s,y)-\f{|x-y|^2}{\eps^2}-\f{|t-s|^2}{\eps^2}-\rho^2(|x|^2+|y|^2),\q j=1,\ldots, J, \l{eq:psi_j}
\ee
for each $t,s\in [0,T]$ and $x,y\in \R^d$, and define the quantity
$$
M^{\eps,\rho}\coloneqq \sup_{j,t,s,x,y}\psi_j^{\eps,\rho}(t,s,x,y).
$$
The upper semicontinuity and boundedness of $U_j-V_j$, along the penalization terms, imply that the supremum is obtained at some point $(j^{\eps,\rho},t^{\eps,\rho},s^{\eps,\rho},x^{\eps,\rho},y^{\eps,\rho})$. Then as in \cite{dumitrescu2015}, { one can find a constant $C$ such that } $|t^{\eps,\rho}-s^{\eps,\rho}|+|x^{\eps,\rho}-y^{\eps,\rho}|\le C\eps$, $|x^{\eps,\rho}|\le \f{C}{\rho}$, and  $|y^{\eps,\rho}|\le \f{C}{\rho}$. Passing to a  subsequence if necessary, we may assume that for each $\rho$, the sequences $\{t^{\eps,\rho}\}_\eps$ and $\{s^{\eps,\rho}\}_\eps$ converge to a common limit $t^\rho$, while  the sequences $\{x^{\eps,\rho}\}_\eps$ and $\{y^{\eps,\rho}\}_\eps$ converge to a common limit $x^\rho$ as $\eps$ tends to $0$. Moreover,  $j^{\eps,\rho}$ lies in a finite set, we may assume $j^{\eps,\rho}=j^\rho$ for all $\eps$. The following lemma gives the convergence  of these sequences, whose proof will be deferred after the proof of Theorem \ref{thm:comparision_s}.

\begin{Lemma}\l{lemma:convlemma}
Extracting  a further subsequence if necessary, we have
\begin{align*}
&\lim_{\eps\to 0}\f{|x^{\eps,\rho}-y^{\eps,\rho}|^2}{\eps^2}=\lim_{\eps\to 0}\f{|t^{\eps,\rho}-s^{\eps,\rho}|^2}{\eps^2}=0, \q 
\lim_{\eps\to 0}U_{j^\rho}(t^{\eps,\rho},x^{\eps,\rho})=U_{j^\rho}(t^\rho,x^\rho)\\
&\lim_{\eps\to 0}V_{j^\rho}(s^{\eps,\rho},y^{\eps,\rho})=V_{j^\rho}(t^\rho,x^\rho),\q \lim_{\rho\to 0} \rho^2|x^\rho|^2=0,\q \lim_{\rho\to 0}\lim_{\eps\to 0}M^{\eps,\rho}=M.
\end{align*}
\end{Lemma}

We now divide our analysis into three cases to establish $M\le 0$.

If there exists a subsequence of $\{t^\rho\}$ such that $t^\rho=0$ for all $\rho$, we then deduce $M\le 0$ along this subsequence by adapting the arguments in \cite{dumitrescu2015} to semicontinuous solutions. 

On the other hand, if  $t^\rho$ is different  from $0$ for all $\rho$, then for any fixed $\rho$ and small enough $\eps$, using Lemma \ref{lemma:system}, which can be proved similarly as Lemma 4.1 in \cite{biswas2010}, we know there exists $j_0^{\eps,\rho}\in \{1,\ldots, J\}$, which for simplicity is still denoted as $j^{\eps,\rho}$, such that $U_{j^{\eps,\rho}}(t^{\eps,\rho},x^{\eps,\rho})>\cM_{j^{\eps,\rho}}U(t^{\eps,\rho},x^{\eps,\rho})$. In other words, at the point 
$(t^{\eps,\rho},s^{\eps,\rho},x^{\eps,\rho},y^{\eps,\rho})$, by considering the $j^{\eps,\rho}$ component of the switching system, we can without loss of generality ignore the term $U_{j^{\eps,\rho}}-\cM_{j^{\eps,\rho}}U$ in the definition of subsolutions and get back to the scalar HJBVI. 

In this case, if we further assume for each $\rho$, there exists a subsequence of $\{ x^{\eps,\rho}\}_\eps$ such that $(U_{j^\rho}-\zeta)(t^{\eps,\rho},x^{\eps,\rho})\le 0$,
then the same arguments in \cite{dumitrescu2015} enables us to derive that $M\le 0$.

Now we come to the final case, where for each $\rho, \eps>0$, we have $(U_{j^\rho}-\zeta)(t^{\eps,\rho},x^{\eps,\rho})> 0$. Applying the nonlocal Jensen-Ishii's lemma as in \cite{dumitrescu2015} enables us to obtain 
$a\in \R$, $\bar{p}, \bar{q}\in \R^d$, and $X,Y\in \R^{d\t d}$
such that it holds for any $\delta>0$ that
\bb\l{eq:ishii}
H_{j^\rho}(t^{\eps,\rho},x^{\eps,\rho},U_{j^\rho}(t^{\eps,\rho},x^{\eps,\rho}),a,\bar{p},X,l_K,l_B)-H_{j^\rho}(s^{\eps,\rho},y^{\eps,\rho},V_{j^\rho}(s^{\eps,\rho},y^{\eps,\rho}),a,\bar{q},Y,l'_K,l'_B)\le 0,
\ee
where  $l_K,l'_K,l_B,l'_B$ are defined as in \cite{dumitrescu2015} for each $\delta>0$ and $H_{j^\rho}(t,x,u,a,p,X,l_1,l_2)$ is given by:
$$
H_{j^\rho} \coloneqq a-\tr(\sigma(\a_{j^\rho},x)(\sigma(\a_{j^\rho},x))^TX)-b(\a_{j^\rho},x)^T p-l_1-f(\a_{j^\rho},t,x,u,\sigma({\a_{j^\rho}},x)^T p,l_2).
$$

We now extend the arguments in \cite{dumitrescu2015} to semicontinuous subsolution $U$ {(resp.\ supersolution $V$)} and argue by contradiction by assuming $M>0$. Then for small enough $\rho,\eps, \delta>0$, we obtain from the monotonicity of $f$ in $u$ that there exists a constant $C>0$ satisfying
\begin{align}
0<&\f{C}{2}M\le CM^{\eps,\rho}\le C(U_{j^\rho}(t^{\eps,\rho},x^{\eps,\rho})-V_{j^\rho}(s^{\eps,\rho},y^{\eps,\rho}))\nb\\
\le& H_{j^\rho}(s^{\eps,\rho},y^{\eps,\rho},U_{j^\rho}(t^{\eps,\rho},x^{\eps,\rho}),a,\bar{q},Y,l'_K,l'_B)-H_{j^\rho}(s^{\eps,\rho},y^{\eps,\rho},V_{j^\rho}(s^{\eps,\rho},y^{\eps,\rho}),a,\bar{q},Y,l'_K,l'_B)\nb\\
=&H_{j^\rho}(s^{\eps,\rho},y^{\eps,\rho},U_{j^\rho}(t^{\eps,\rho},x^{\eps,\rho}),a,\bar{q},Y,l'_K,l'_B)-H_{j^\rho}(t^{\eps,\rho},x^{\eps,\rho},U_{j^\rho}(t^{\eps,\rho},x^{\eps,\rho}),a,\bar{p},X,l'_K,l'_B)\l{eq:line1}\\
+&H_{j^\rho}(t^{\eps,\rho},x^{\eps,\rho},U_{j^\rho}(t^{\eps,\rho},x^{\eps,\rho}),a,\bar{p},X,l'_K,l'_B)
-H_{j^\rho}(t^{\eps,\rho},x^{\eps,\rho},U_{j^\rho}(t^{\eps,\rho},x^{\eps,\rho}),a,\bar{p},X,l_K,l_B)
\l{eq:line2}\\
+&
H_{j^\rho}(t^{\eps,\rho},x^{\eps,\rho},U_{j^\rho}(t^{\eps,\rho},x^{\eps,\rho}),a,\bar{p},X,l_K,l_B)-H_{j^\rho}(s^{\eps,\rho},y^{\eps,\rho},V_{j^\rho}(s^{\eps,\rho},y^{\eps,\rho}),a,\bar{q},Y,l'_K,l'_B),
\l{eq:line3}
\end{align}
from which, by expanding \eqref{eq:line1}, using  the fact that $f$ is Lipschitz continuous and monotone in $k$ for \eqref{eq:line2},
 and applying \eqref{eq:ishii} to \eqref{eq:line3}, we can derive that
\begin{align*}
 0&\le CM/2\le l_K-l'_K+C(l_B-l'_B)+[b(\a_{j^\rho},x^{\eps,\rho})^T \bar{p}-b(\a_{j^\rho},y^{\eps,\rho})^T \bar{q}]\\
 +&\f{1}{2}
\big[\tr(\sigma(\a_{j^\rho},x^{\eps,\rho})(\sigma(\a_{j^\rho},x^{\eps,\rho}))^TX-\sigma(\a_{j^\rho},y^{\eps,\rho})(\sigma(\a_{j^\rho},y^{\eps,\rho}))^TY)\big]\\
+&f(\a_{j^\rho}, t^{\eps,\rho},x^{\eps,\rho},U_{j^\rho}(t^{\eps,\rho},x^{\eps,\rho}),
\sigma(\a_{j^\rho},x^{\eps,\rho})^T\bar{p},l'_B)
-f(\a_{j^\rho}, s^{\eps,\rho},x^{\eps,\rho},U_{j^\rho}(t^{\eps,\rho},x^{\eps,\rho}),
\sigma(\a_{j^\rho},x^{\eps,\rho})^T\bar{p},l'_B)\\
+&f(\a_{j^\rho}, s^{\eps,\rho},x^{\eps,\rho},U_{j^\rho}(t^{\eps,\rho},x^{\eps,\rho}),
\sigma(\a_{j^\rho},x^{\eps,\rho})^T\bar{p},l'_B)
-f(\a_{j^\rho}, s^{\eps,\rho},y^{\eps,\rho},U_{j^\rho}(t^{\eps,\rho},x^{\eps,\rho}),
\sigma(\a_{j^\rho},x^{\eps,\rho})^T\bar{p},l'_B)\\
+&f(\a_{j^\rho}, s^{\eps,\rho},y^{\eps,\rho},U_{j^\rho}(t^{\eps,\rho},x^{\eps,\rho}),
\sigma(\a_{j^\rho},x^{\eps,\rho})^T\bar{p},l'_B)
-f(\a_{j^\rho}, s^{\eps,\rho},y^{\eps,\rho},U_{j^\rho}(t^{\eps,\rho},x^{\eps,\rho}),
\sigma(\a_{j^\rho},y^{\eps,\rho})^T\bar{q},l'_B)\\
\le& 
l_K-l'_K+C(l_B-l'_B)+\f{1}{2}
\big[\tr(\sigma(\a_{j^\rho},x^{\eps,\rho})(\sigma(\a_{j^\rho},x^{\eps,\rho}))^TX-\sigma(\a_{j^\rho},y^{\eps,\rho})(\sigma(\a_{j^\rho},y^{\eps,\rho}))^TY)\big]\\
+&f(\a_{j^\rho}, t^{\eps,\rho},x^{\eps,\rho},U_{j^\rho}(t^{\eps,\rho},x^{\eps,\rho}),
\sigma(\a_{j^\rho},x^{\eps,\rho})^T\bar{p},l'_B)
-f(\a_{j^\rho}, s^{\eps,\rho},x^{\eps,\rho},U_{j^\rho}(t^{\eps,\rho},x^{\eps,\rho}),
\sigma(\a_{j^\rho},x^{\eps,\rho})^T\bar{p},l'_B)\\
+&m_R(|x^{\eps,\rho}-y^{\eps,\rho}|(1+|\sigma(\a_{j^\rho},x^{\eps,\rho})^T \bar{p}|)\\
+&[b(\a_{j^\rho},x^{\eps,\rho})^T \bar{p}-b(\a_{j^\rho},y^{\eps,\rho})^T \bar{q}]
+[\sigma(\a_{j^\rho},x^{\eps,\rho})^T \bar{p}-\sigma(\a_{j^\rho},y^{\eps,\rho})^T \bar{q}],
\end{align*}
{for some  $R \geq \frac{C}{\rho} \vee ||U_{j^\rho}||_{\infty}.$}

Then noticing the estimates derived in \cite{dumitrescu2015} for each term on the right-hand side of the above expression are uniform in the control $\a_{j^\rho}$, and successively passing $\delta,\eps$ and $\rho$ to $0$, we deduce that $0<M\le 0$, which leads to a contradiction. Thus we conclude $M\le 0$ and complete the proof.
\end{proof}

\begin{proof}[Proof of Lemma \ref{lemma:convlemma}]
For each $\rho>0$ and $j=1,\ldots, J$, we introduce the functions $\hat{U}_j^\rho(t,x)=U_j(t,x)-\eta^2|x|^2$ and $\hat{V}_j^\rho(t,x)=V_j(t,x)-\eta^2|x|^2$. Then we define
$$
M^\rho=\sup_{j,t,x}(\hat{U}_j^\rho-\hat{V}_j^\rho),
$$
which is attained at some point $(\hat{j}^\rho,\hat{t}^\rho,\hat{x}^\rho)$. Recall that for any $\rho$, we can assume without loss of generality that 
$\{(t^{\eps,\rho},s^{\eps,\rho},x^{\eps,\rho},y^{\eps,\rho})\}_\eps$ converges to 
$(t^{\rho},t^{\rho},x^{\rho},x^{\rho})$ as $\eps\to 0$ and $j^{\eps,\rho}=j^\rho$ for all $\eps$. Then the definition of $M^{\eps,\rho}$ gives us that
\begin{align}\l{eq:Mrho-Meps}
M^\rho=(\hat{U}_{\hat{j}^\rho}^\rho-\hat{V}_{\hat{j}^\rho}^\rho)(\hat{t}^\rho,\hat{x}^\rho)\le M^{\eps,\rho}
=&U_{j^\rho}(t^{\eps,\rho},x^{\eps,\rho})-V_{j^\rho}(s^{\eps,\rho},y^{\eps,\rho})\nb\\
&-\f{|x^{\eps,\rho}-y^{\eps,\rho}|^2}{\eps^2}-\f{|t^{\eps,\rho}-s^{\eps,\rho}|^2}{\eps^2}-\rho^2(|x^{\eps,\rho}|^2+|y^{\eps,\rho}|^2).
\end{align}
Define $\bar{l}_\rho=\limsup_{\eps\to 0}\f{|x^{\eps,\rho}-y^{\eps,\rho}|^2}{\eps^2}$ and $\ul{l}_\rho=\liminf_{\eps\to 0}\f{|x^{\eps,\rho}-y^{\eps,\rho}|^2}{\eps^2}$, then we obtain from 
\eqref{eq:Mrho-Meps} that
\begin{align*}
0\le \ul{l}_\rho\le\bar{l}_\rho&\le
 \limsup_{\eps\to 0}\big(U_{j^\rho}(t^{\eps,\rho},x^{\eps,\rho})-V_{j^\rho}(s^{\eps,\rho},y^{\eps,\rho})-\rho^2(|x^{\eps,\rho}|^2+|y^{\eps,\rho}|^2)\big)-(\hat{U}_{\hat{j}^\rho}^\rho-\hat{V}_{\hat{j}^\rho}^\rho)(\hat{t}^\rho,\hat{x}^\rho)\\
&\le (\hat{U}_{j^\rho}^\rho-\hat{V}_{j^\rho}^\rho)(t^\rho,x^\rho)-(\hat{U}_{\hat{j}^\rho}^\rho-\hat{V}_{\hat{j}^\rho}^\rho)(\hat{t}^\rho,\hat{x}^\rho)\le 0,
\end{align*}
where we have used the  semicontinuity of $U_{j^\rho}$ and $V_{j^\rho}$. Similarly, we can derive $\lim_{\eps\to 0}\f{|t^{\eps,\rho}-s^{\eps,\rho}|^2}{\eps^2}=0$, which along with \eqref{eq:Mrho-Meps} implies that $\lim_{\eps\to 0}M^{\eps,\rho}=M^\rho$. The fact that $\lim_{\rho\to 0}M^\rho=M$ can be shown as in \cite{dumitrescu2015}.

Let us now prove $\lim_{\rho\to 0}\rho^2|x^\rho|^2=0$. It holds for each $\rho>0$ that
\bb\l{eq:Mrho}
M^\rho=\lim_{\eps\to 0}M^{\eps,\rho}=\limsup_{\eps\to 0}M^{\eps,\rho}\le 
U_{j^\rho}(t^{\rho},x^{\rho})-V_{j^\rho}(t^{\rho},x^{\rho})-2\rho^2|x^{\rho}|^2\le M^\rho,
\ee
and hence all inequalities in the above expression are in fact equalities. Thus we have
\begin{eqnarray*}
M^{\rho/2}&=&\sup_{j,t,x}\bigg[U_j(t,x)-V_j(t,x)-2(\f{\rho}{2})^2|x|^2\bigg]\ge U_{j^\rho}(t^{\rho},x^{\rho})-V_{j^\rho}(t^{\rho},x^{\rho})-2\rho^2|x^{\rho}|^2+\f{3}{2}\rho^2|x^{\rho}|^2 \\
&=&
M^\rho+\f{3}{2}\rho^2|x^{\rho}|^2,
\end{eqnarray*}
which implies $0\le \limsup_{\rho\to 0}\rho^2|x^{\rho}|^2\le \limsup_{\rho\to 0}\f{2}{3}(M^{\rho/2}-M^{\rho})=0$.

Finally, we obtain from \eqref{eq:Mrho} and $\lim_{\eps\to0}\f{|t^{\eps,\rho}-s^{\eps,\rho}|^2}{\eps^2}=\lim_{\eps\to0}\f{|x^{\eps,\rho}-y^{\eps,\rho}|^2}{\eps^2}=0$ that 
we have $\lim_{\eps\to 0}U_{j^\rho}(t^{\eps,\rho},x^{\eps,\rho})-V_{j^\rho}(s^{\eps,\rho},y^{\eps,\rho})=U_{j^\rho}(t^{\rho},x^{\rho})-V_{j^\rho}(t^{\rho},x^{\rho})$, which together the  semicontinuity of $U_{j^\rho}$ and $V_{j^\rho}$ implies
 $\limsup_{\eps\to 0}U_{j^\rho}(t^{\eps,\rho},x^{\eps,\rho})=U_{j^\rho}(t^{\rho},x^{\rho})$ and 
 $\liminf_{\eps\to 0}V_{j^\rho}(s^{\eps,\rho},y^{\eps,\rho})=V_{j^\rho}(t^{\rho},x^{\rho})$. By extracting further subsequences if necessary, we complete our proof.
\end{proof}


\begin{Lemma}\l{lemma:system}
Let $U$ (resp.~V) be a bounded subsolution (resp.\ supersolution) of \eqref{eq:hjbvi_s}. For any given $\eps,\rho>0$, we consider the function
$\psi^{\eps,\rho}_{j}(t,s, x,y)$ as defined in \eqref{eq:psi_j} 
and $M^{\eps,\rho}=\sup_{j,t,s,x,y} \psi^{\eps,\rho}_i(t,s,x,y)$. 
If there exists an index $j^{\eps, \rho}$ and  a point $(t^{\eps, \rho},s^{\eps, \rho},x^{\eps, \rho},y^{\eps, \rho})\in (0,T]^2\t\R^{2d}$ such that $\psi_{j^{\eps, \rho}}(s^{\eps, \rho},t^{\eps, \rho},x^{\eps, \rho},y^{\eps, \rho})=M^{\eps,\rho}$, then there exists an index $j_0^{\eps, \rho} \in \{1,\ldots, J\}$ such that
$$
\textnormal{$\psi_{j^{\eps, \rho}_0}(s^{\eps, \rho},t^{\eps, \rho},x^{\eps, \rho},y^{\eps, \rho})=M^{\eps,\rho}$~~ and  ~~
$U_{j^{\eps, \rho}_0}(t^{\eps, \rho},x^{\eps, \rho})>\mathcal{M}_{j^{\eps, \rho}_0}(t^{\eps, \rho},x^{\eps, \rho})$}.
$$
\end{Lemma}

\section{Truncation of singular measures}\l{sec:jumps}
A possible way to work with
a nonsingular jump measure is to introduce a Backward SDE with a modified driver and an approximative jump-diffusion dynamics where the small jumps part has been substituted  by a rescaled diffusion coefficient of  the Brownian motion $W$.

More precisely, we adopt the same probability space as introduced in Section \ref{sec:formulation}, which supports the Brownian motion process $W$ and the independent Poisson measure $N(dt,de)$. For a given jump truncation size $\eps>0$, we define a modified diffusion coefficient  $\tilde{\sigma}^\a$ as in \eqref{eq:small_diff},
and also introduce a modified driver 
$f^\varepsilon(\alpha,t,x,y,z,k):=\hat{f}(\alpha,t, x,y,z,\int_{|e| \geq \varepsilon} k(e)\gamma(\alpha,x,e)\nu(de)),$ where the function $\hat{f}$ is given in Assumption \ref{assum:coeff}.

For any given  initial state $x\in \R^d$,  control $\a\in \cA_t^t$ and $\tau\in \cT_t^t$, we consider the modified controlled jump-diffusion process 
 $(X^{\eps,\a,t, x}_s)_{ t \le s\le T}$ satisfying the following SDE: for each $s\in [t,T]$,
\bb\l{eq:sde_m}
X^{\eps,\a,t,x}_s=x+\int_t^s b(\a_v,X^{\eps,\a,t,x}_v)\,dv+\int_t^s \tilde{\sigma} (\a_v,X^{\a,t,x}_v)\,dW^t_v+\int_t^s\int_{|e|>\eps} \eta(\a_v,X^{\a,t,x}_v,e)\,\tilde{N}^t(dv,de),
\ee
and the BSDE with the modified  controlled driver $f^\eps(\a_s,s,X_s^{\a,t,x},y,z,k)$:
\bb\l{eq:bsde_m}
\begin{cases}
-Y^{\eps,\a,t,x}_{s,\tau}=f^\eps(\a_s,s,X_s^{\eps,\a,t,x},Y^{\eps,\a,t,x}_{s,\tau},Z^{\eps,\a,t,x}_{s,\tau},K^{\eps,\a,t,x}_{s,\tau})ds-Z^{\eps,\a,t,x}_{s,\tau}dW^t_s-\int_{E} K^{\eps,\a,t,x}_{s,\tau}(e)\,\tilde{N}^t(ds,de),\\
Y^{\eps,\a,t,x}_{\tau,\tau}=\xi(\tau,X^{\eps,\a,t,x}_\tau).
\end{cases}
\ee
The coefficients of the above SDE and BSDE satisfy Assumption \ref{assum:coeff}, and therefore the equations are well-posed.

Now we are ready to state the modified mixed optimal stopping and control problem.  
For each initial 
time $t\in[0,T]$ and initial state $x\in \R^d$, we consider the following value function:
\bb\l{eq:nonlinear_control_m}
u^\eps(t,x)= \sup_{\tau\in \cT_t^t}\sup_{\a\in \cA_t^t} \cE^{f^{\a,\eps}}_{t,\tau}[\xi(\tau, X^{\eps,\a, t,x}_\tau)],
\ee
subject to the controlled SDE \eqref{eq:sde_m}, where the nonlinear expectation is induced by \eqref{eq:bsde_m}.

Let us first show the following uniform convergence result of the forward component $X^{\varepsilon, \alpha,t,x}$ towards $X^{\alpha,t,x}$ when $\varepsilon$ tends to $0$.

\begin{Lemma}

For each $t\in [0,T]$, $x \in \mathbb{R}^d$ and $\alpha \in \mathcal{A}_t^t$ it holds that 
\begin{align}\l{eq:Xuniform}
\mathbb{E}\left[\sup_{t \leq s \leq T} |X_s^{\varepsilon, \alpha,t,x}-X_s^{\alpha,t,x}|^2\right] \leq C \kappa(\varepsilon) ,
\end{align}
with $\kappa(\varepsilon):=\int_{|e|\leq \eps}(1 \wedge |e|^2)\nu(de)$ and $C$ a constant independent of $\alpha$ and $\varepsilon$. 
\end{Lemma}

\begin{proof}
Fix  $\alpha \in \mathcal{A}_t^t$ and $v \in [t,T]$. We have:
\begin{align*}
\mathbb{E}\left[\sup_{t \leq u \leq v} |X_u^{\varepsilon, \alpha,t,x}-X_u^{\alpha,t,x}|^2\right] &\leq C \mathbb{E}\left[ \sup_{t \leq u \leq v} \left(\int_t^u (b(\alpha_s, X_s^{\varepsilon, \alpha,t,x})-b(\alpha_s, X_s^{\alpha,t,x}))ds\right)^2 \right] \nonumber \\
&\hspace{-1cm} + C \mathbb{E}\left[ \sup_{t \leq u \leq v} \left(\int_t^u (\Tilde{\sigma}(\alpha_s, X_s^{\varepsilon, \alpha,t,x})-\sigma(\alpha_s, X_s^{\alpha,t,x}))dW_s\right)^2 \right] \nonumber \\
&\hspace{-1cm}+C \mathbb{E}\left[ \sup_{t \leq u \leq v} \left(\int_t^u \int_{|e| > \varepsilon}(\eta(\alpha_s, X_s^{\varepsilon, \alpha,t,x},e)-\eta(\alpha_s, X_s^{\alpha,t,x},e))\Tilde{N}(ds,de)\right)^2 \right] \nonumber \\
&\hspace{-1cm}+C \mathbb{E}\left[ \sup_{t \leq u \leq v} \left(\int_t^u \int_{|e| \leq \varepsilon}(\eta(\alpha_s, X_s^{\alpha,t,x},e))\Tilde{N}(ds,de)\right)^2 \right],
\end{align*}
where $C$ is a constant independent of $\alpha$.
The Burkholder-Davis-Gundy inequality, together with the Lipschitz assumptions on the coefficients $b, \sigma, \eta$ (see Assumption \ref{assum:coeff}) lead to:
\begin{align*}
\mathbb{E}\left[\sup_{t \leq u \leq v} |X_u^{\varepsilon, \alpha,t,x}-X_u^{\alpha,t,x}|^2\right] &\leq C \mathbb{E}\left[  \int_t^v \left(X_s^{\varepsilon, \alpha,t,x}-X_s^{\alpha,t,x}\right)^2ds \right] \nonumber \\
&\hspace{-0cm}+C \mathbb{E}\left[ \left(\int_t^v \left(X_s^{\varepsilon, \alpha,t,x}-X_s^{\alpha,t,x}\right)^2ds\right) \left(\int_{|e| > \varepsilon} (1 \wedge |e|^2) \nu(de)\right)\right] \nonumber \\
&\hspace{-0cm}+C \mathbb{E}\left[  \int_{|e| \leq \varepsilon}\left(1 \wedge |e|\right)^2 \nu(de)\right] \nonumber \\
&\hspace{-1cm} \leq C \mathbb{E}\left[  \int_t^v \left(\sup_{t \leq u \leq s}\left(X_u^{\varepsilon, \alpha,t,x}-X_u^{\alpha,t,x}\right)^2\right)ds \right]+C \bigg( \int_{|e| \leq \varepsilon} (1 \wedge |e|^2) \nu(de)\bigg),
\end{align*}
where the last inequality follows by the integrability assumption on the measure $\nu$.
Then we obtain the desired result \eqref{eq:Xuniform}  from the Gronwall's inequality.
%
\end{proof}

Using the above estimate, we now show the convergence of the value function $u^\varepsilon$ towards $u$.

\begin{Lemma}
For each $t \in [0,T]$ and $x \in \mathbb{R}^d$ we have
\begin{align}\l{eq:u_eps_conv}
\left| u^\varepsilon (t,x)-u(t,x)\right|  \leq C (\kappa(\varepsilon))^{\frac{1}{4}},
\end{align}
with $\kappa(\varepsilon):=\int_{|e|\leq \eps}(1 \wedge |e|^2)\nu(de)$
and  $C$ a constant independent of $\varepsilon$.
\end{Lemma}

\begin{proof}
Fix $t \in [0,T]$ and $x \in \mathbb{R}^d$. The definitions of $u^\varepsilon$ and  $u$ imply that
\begin{align}\label{eqq1}
|u^\varepsilon(t,x)-u(t,x)|^2 &= \big|\sup_{\alpha \in \mathcal{A}_t^t } \sup_{\tau \in \mathcal{T}_t^t } \mathcal{E}^{f^{\varepsilon,\alpha}}_{t,\tau}\left[\zeta(\tau, X_\tau^{\varepsilon, \alpha,t,x})\right]-\sup_{\alpha \in \mathcal{A}_t^t } \sup_{\tau \in \mathcal{T}_t^t } \mathcal{E}^{f^{\alpha}}_{t,\tau}\left[\zeta(\tau, X_\tau^{\alpha,t,x})\right] \big|^2 \nonumber \\
&  \leq \sup_{\alpha \in \mathcal{A}_t^t } \sup_{\tau \in \mathcal{T}_t^t} \left|\mathcal{E}^{f^{\varepsilon,\alpha}}_{t,\tau}\left[\zeta(\tau, X_\tau^{\varepsilon, \alpha,t,x})\right]- \mathcal{E}^{f^{\alpha}}_{t,\tau}\left[\zeta(\tau, X_\tau^{\alpha,t,x})\right] \right|^2. 
\end{align}
Recall that, since $\alpha \in \mathcal{A}_t^t$ and $\tau \in \mathcal{T}_t^t$, $\bigg |\mathcal{E}^{f^{\varepsilon,\alpha}}_{t,\tau}\left[\zeta(\tau, X_\tau^{\varepsilon, \alpha,t,x})\right]- \mathcal{E}^{f^{\alpha}}_{t,\tau}\left[\zeta(\tau, X_\tau^{\alpha,t,x})\right] \bigg|$ is deterministic. By the a priori estimates on the spread between the first component of the solutions of two BSDEs with jumps (see Proposition A.4. in \cite{quenezsulem2016}), we derive that there exist $\beta>0$ and $\eta>0$ independent on $\tau \in \mathcal{T}_t^t$ and $\alpha \in \mathcal{A}_t^t$, such that 
\begin{align*}
 &\left|\mathcal{E}^{f^{\varepsilon,\alpha}}_{t,\tau}\left[\zeta(\tau, X_\tau^{\varepsilon, \alpha,t,x})\right]- \mathcal{E}^{f^{\alpha}}_{t,\tau}\left[\zeta(\tau, X_\tau^{\alpha,t,x})\right] \right|^2 \leq \mathbb{E}\left[e^{\beta(\tau-t)} \left(\zeta(\tau, X_\tau^{\varepsilon, \alpha,t,x})-\zeta(\tau, X_\tau^{\alpha,t,x})\right)^2  \right] \nonumber \\
&+\eta \mathbb{E} \left[\int_t^\tau e^{\beta (s-t)} \left(f(s,\alpha_s,X_s^{\alpha,t,x},Y_{s,\tau}^{\alpha,t,x},Z_{s,\tau}^{\alpha,t,x},K_{s,\tau}^{\alpha,t,x})-f^\varepsilon(s,\alpha_s,X_s^{\varepsilon,\alpha,t,x},Y_{s,\tau}^{\alpha,t,x},Z_{s,\tau}^{\alpha,t,x},K_{s,\tau}^{\alpha,t,x})\right)^2 ds\right] \nonumber \\
& \leq C \bigg( \mathbb{E}\bigg[\sup_{t \leq u \leq T} |X_u^{\varepsilon, \alpha,t,x}-X_u^{\alpha,t,x}|^2  \bigg]+\mathbb{E}\bigg[ \int_t^\tau \bigg(\int_{|e| \leq \eps} K_{s,\tau}^{\alpha,t,x}(e)\gamma(X_s^{\alpha,t,x},e)\nu(de)\bigg)^2ds\bigg]\bigg) \nonumber \\
& +C \mathbb{E}\bigg[ \int_t^\tau \bigg(\int_{|e|> \varepsilon} K_{s,\tau}^{\alpha,t,x}(e)(\gamma(X_s^{\alpha,t,x},e)-\gamma(X_s^{\varepsilon,\alpha,t,x},e))\nu(de)\bigg)^2 ds\bigg] \nonumber \\
& \leq C \left( \mathbb{E}\bigg[\sup_{t \leq u \leq T} |X_u^{\varepsilon, \alpha,t,x}-X_u^{\alpha,t,x}|^2\bigg]+ \mathbb{E}\bigg[ \int_t^\tau \!\!\! \bigg(\!\int_{|e| \leq \eps} \!(K_{s,\tau}^{\alpha,t,x})^2(e) \nu(de)\bigg) \!\bigg(\!\int_{|e| \leq \eps} \!\! \gamma^2(X_s^{\alpha,t,x},e)\nu(de)\bigg)ds\bigg]\right) \nonumber \\
&+C \mathbb{E}\bigg[\! \int_t^\tau \bigg(\!\int_{|e|> \varepsilon} (K_{s,\tau}^{\alpha,t,x})^2(e)\nu(de)\bigg)\bigg(\int_{|e|> \varepsilon}(\gamma(X_s^{\alpha,t,x},e)-\gamma(X_s^{\varepsilon,\alpha,t,x},e))^2\nu(de)\bigg)ds\bigg],
\end{align*}
where $C$ is a constant independent on $\alpha$ and $\tau$, only depending on $\beta$, $\eta$, $T$ and the Lipschitz constant of $f$.
Now, the uniform boundness of $\zeta$, $g$ and $f$ with respect to $\alpha$ and $\tau$ (see Assumption \ref{assum:coeff}), together with the a priori estimates for $L^p$ solutions of BSDEs  with $p=2$ and $p=4$ (see Proposition 2 in \cite{Kruse2016}) gives us an uniform control on the $\mathbb{H}^2_{t,\nu}$ (resp. $\mathbb{H}^4_{t,\nu}$) norm of $K_{\cdot,\tau}^{\alpha,t,x}$ (which only depends on the bounds of $\zeta$, $g$, $f$ and $T$). Using this result and the assumptions on the map $\gamma$ (Assumption \ref{assum:coeff}), we derive that
 there exists a constant $C$ independent on $\tau$ and $\alpha$ such that
\begin{align*}
\left|\mathcal{E}^{f^{\varepsilon,\alpha}}_{t,\tau}\left[\zeta(\tau, X_\tau^{\varepsilon, \alpha,t,x})\right]- \mathcal{E}^{f^{\alpha}}_{t,\tau}\left[\zeta(\tau, X_\tau^{\alpha,t,x})\right] \right| \leq C (\kappa(\varepsilon))^{\frac{1}{4}}.
\end{align*}
We now take the supremum over $\alpha$ and $\tau$ and using $\eqref{eqq1}$ we obtain \eqref{eq:u_eps_conv} .
\end{proof}

In the following theorem, we  show that $u^\varepsilon$ is  the unique viscosity solution of the (backward) HJBVI equation \eqref{eq:hjbvi_trun} introduced in Section \ref{sec:scheme}.
\begin{Theorem}
The function $u^\varepsilon$ defined by $\eqref{eq:nonlinear_control_m}$ is the unique viscosity solution of the obstacle problem: $u(T,x)=g(x)$ for $x\in\R^d$ and 
\begin{align*}
\min\{u^\varepsilon-\zeta, \inf_{\alpha \in \textbf{A}}(-u^\eps_t-L^\a_\eps u^\eps-f^{\eps,\a}(\x,u^\eps, (\tilde{\sigma}^\a)^T Du^\eps, B^\a u^\eps))\} = 0, \q \x\in [0,T)\t\R^d.
\end{align*}
\end{Theorem}

\begin{proof}
{Due to the compactness of the set $\textbf{A}$, the proof is similar to the one given in \cite{dumitrescu2015} in the case without controls.}
\end{proof}

\begin{Remark}
Contrary to the case without control and optimal stopping studied in \cite{Giulia2013}, it is not clear that one can use a different approximation of the forward backward system by introducing an independent Brownian motion scaled with the standard deviation of small jumps. Indeed, the equations would be well-posed in an enlarged filtration $\mathbb{G}$, but the control process is $\mathbb{F}$-predictable, with $\mathbb{F} \subset \mathbb{G}$, which leads to difficulties in the derivation of the dynamic programming principle.
\end{Remark}


\end{document}